\documentclass[mathpazo]{cicp}
\usepackage{amssymb}
\usepackage{amsfonts, amssymb}
\usepackage{amsmath}
\usepackage{amsthm}
\newtheorem{lem}{Lemma}
\usepackage{multirow}
\usepackage{diagbox}
\usepackage{booktabs}
\usepackage{longtable}
\usepackage{color}
\usepackage{hyperref}
\usepackage{subfigure}
\usepackage{textcomp}
\usepackage{lineno}
\let\oldequation\equation
\let\oldendequation\endequation

\renewenvironment{equation}
{\linenomathNonumbers\oldequation}
{\oldendequation\endlinenomath}

\hypersetup{hypertex=true,
	colorlinks=true,
	linkcolor=blue,
	anchorcolor=blue,
	citecolor=blue}
\usepackage{float} 
\newtheorem{thm}{Theorem}

\begin{document}
\title{A Kolmogorov High Order Deep Neural Network for High Frequency Partial Differential Equations in High Dimensions}



\author[Zhang Y Q et.~al.]{Yaqin Zhang\affil{1,2}\comma\footnotemark[2],
Ke Li\affil{3}\comma\footnotemark[2], Zhipeng Chang\affil{4}, Xuejiao Liu\affil{2}, Yunqing Huang\affil{1}\comma\footnotemark[1], and Xueshuang Xiang\affil{2}\comma\corrauth}
\footnotetext[2]{These authors contributed equally to this work.}
\address{\affilnum{1}\ School of Mathematics and Computational Science,
Xiangtan University,
Xiangtan 411105, P.R. China. \\
\affilnum{2}\ Qian Xuesen Laboratory of Space Technology,
China Academy of Space Technology, Beijing 100094, P.R. China.\\
\affilnum{3}\ Information Engineering University, Zhengzhou 450001, P.R. China.\\
\affilnum{4}\ School of Mathematics and Statistics, Wuhan University, Wuhan 430072, P.R. China.}

\emails{
{\tt huangyq@xtu.edu.cn} (Y.~Huang), {\tt xiangxueshuang2023@163.com} (X.~Xiang)}


\begin{abstract}
This paper proposes a Kolmogorov high order deep neural network (K-HOrderDNN)  for solving high-dimensional partial differential equations (PDEs), which improves the high order deep neural networks (HOrderDNNs). HOrderDNNs have been demonstrated to outperform conventional DNNs for high frequency problems by introducing a nonlinear transformation layer consisting of  $(p+1)^d$ basis functions. However, the number of basis functions grows exponentially with the dimension $d$, which results in the curse of dimensionality (CoD). Inspired by the Kolmogorov superposition theorem (KST), which expresses a multivariate function as superpositions of univariate functions and addition,  K-HOrderDNN utilizes a HOrderDNN to efficiently approximate univariate inner functions instead of directly approximating the multivariate function, reducing the number of introduced basis functions to $d(p+1)$. We theoretically demonstrate that CoD is mitigated when target functions belong to a dense subset of continuous multivariate functions. Extensive numerical experiments show that: for high-dimensional problems  ($d$=10, 20, 50) where HOrderDNNs($p>1$) are intractable, K-HOrderDNNs($p>1$) exhibit remarkable performance. Specifically, when $d=10$, K-HOrderDNN($p=7$) achieves an error of 4.40E-03, two orders of magnitude lower than that of HOrderDNN($p=1$) (see Table \ref{tab:table14}); for high frequency problems, K-HOrderDNNs($p>1$) can achieve higher accuracy with fewer parameters and faster convergence rates compared to HOrderDNNs (see Table \ref{tab:table12}). 
\end{abstract}

\ams{68T99, 35Q68, 65N99}
\keywords{Deep neural network, Kolmogorov Superposition Theorem, high-dimensional and high-frequency PDEs.}

\maketitle

\section{Introduction}
\label{Introduction}
In recent years, deep learning methods for solving partial differential equations (PDEs) have attracted widespread attention \cite{Han_Jentzen_E_2018}. A series of novel methods, such as the deep Ritz method (DRM) \cite{yu2018deep}, Physical Information Neural Networks (PINNs) \cite{raissi2019physics,qian2023physics}, the deep Galerkin method (DGM) \cite{sirignano2018dgm} and Weak Adversarial Neural Networks (WAN) \cite{zang2020weak}, have been proposed and demonstrated the vast potential of deep neural networks (DNNs) in solving various PDEs. The basic idea of these methods is reformulating a PDE problem as an optimization problem and training a DNN to approximate the solution of the PDE by minimizing the corresponding loss function. Compared to traditional mesh-dependent methods such as Finite Element Methods (FEMs), these deep learning-based numerical methods show great flexibility and potential for solving complex PDEs defined in high dimensions and irregular domains. Unfortunately, challenges still exist despite their early success, especially when dealing with high-frequency problems. As revealed by frequency principle or spectral bias \cite{xu2018understanding, xu2019frequency}, PINNs exhibit different learning behaviours among different frequencies, with low-frequency components being prioritized and high-frequency components hard to capture,  ultimately leading to difficulties in achieving stable training and accurate predictions in high-frequency problems.
\par
To address this challenge, a series of extensions to the vanilla PINN have been proposed to boost the performance of PINNs from various aspects, such as PhaseDNN \cite{cai2020phase}, MscaleDNN and its variants \cite{cai2019multi, liu2020multi, li2020dnn}, cFPCT-DNN \cite{zhang2023learning}, PIRBN \cite{bai2023physics}, etc. In particular, \cite{chang2022high} develops the High Order Deep Neural Network (HOrderDNN), which incorporates high order idea from FEMs into conventional neural networks by introducing a nonlinear transformation layer determined by high order basis functions. As demonstrated in \cite{chang2022high}, HOrderDNN($p$) can directly reproduce polynomials in $\mathcal{Q}_p (\mathbb{R}^d)$, efficiently capture the high frequency information in target functions, and obtain greater approximation capability, additional efficiency and higher accuracy in solving high frequency problems compared to PINN. Furthermore,  \cite{Chang_Li_Zou_Xiang} develops HOrderDeepDDM by combining HOrderDNN with the domain decomposition method for solving high frequency interface problems. However, with the powerful approximation capability in HOrderDNN, there are also some limitations. Specifically, the nonlinear transformation layer incorporated in HOrderDNN requires $(p+1)^d$ tensor product basis functions to reproduce any multivariate polynomial function $f(x) \in \mathcal{Q}_p (\mathbb{R}^d)$. As the dimension $d$ increases, the number of basis functions $(p+1)^d$ grows exponentially, resulting in HOrderDNN suffering the curse of dimensionality (CoD).
\par
In this paper, we continue this line of research for solving high frequency problems and propose a Kolmogorov high order deep neural network, named K-HOrderDNN, to address the issue of CoD in HOrderDNN. Drawing inspiration from the famous Kolmogorov Superposition Theorem (KST), we reconstruct the nonlinear transform layer employed in HOrderDNN to avoid the use of the tensor product basis functions, which are the primary cause of CoD. The Kolmogorov Superposition Theorem was first proposed by Kolmogorov \cite{kolmogorov1957representation} in 1957, and several improved versions have been proposed subsequently. The main idea behind KST is that any continuous multivariate function can be decomposed into linear combinations and compositions of a small number of univariate functions, which provides a framework for representing complex multivariate functions by univariate functions, and offers a new insight into function representation and approximation in high dimensions. Inspired by this remarkable insight, we turn to efficiently approximate univariate functions instead of approximating a multivariate function directly, and assemble these univariate approximations into a multivariate function following the framework of KST. Considering the superiority of \textquotedblleft high order\textquotedblright idea in handling high-frequency problems, high order univirate basis functions are introduced into the subnetwork to 
approximate inner univiriate functions in the representation of KST. We remark that only $(p+1)d$ univariate basis functions are introduced in K-HOrderDNN rather than $(p+1)^d$ tensor product basis functions as in HOrderDNN, avoiding the exponential growth. Furthermore, we theoretically examine the approximation power of K-HOrderDNN and demonstrate that CoD is avoided for a dense subset of continuous multivariate functions.
\par
Extensive numerical experiments are conducted to demonstrate the performance of the proposed K-HOrderDNN. The proposed K-HOrderDNN is compared with PINN and HOrderDNNs on high frequency function fitting problems,   high frequency Poisson equations and Helmholtz equations, especially in high dimensions. Our numerical results indicate that the proposed K-HOrderDNN has the following advantages: 
\par
1. K-HOrderDNN($p$) further enhances the advantages of HOrderDNN($p$) over PINN for addressing high frequency problems. K-HOrderDNNs($p>1$) achieve faster convergence and smaller relative errors than HOrderDNNs($p>1$), both of which are significantly better than PINN, as illustrated in Fig. \ref{fig:8} and Fig. \ref{fig:14}. Especially, K-HOrderDNN($p=9$) can attain relative errors approximately two orders of magnitude smaller than those of PINN while using only about one-third of the parameters, as shown in Table \ref{tab:table7} and Table \ref{tab:table12}. In addition, K-HOrderDNNs($p>1$) resemble HOrderDNNs($p>1$) in that they tend to learn low frequency components and high frequency components almost simultaneously when approximating functions with multiple frequencies, as shown in Fig. \ref{fig:7}. 
\par
2. K-HOrderDNN($p$) addresses the parameter explosion issue faced by HOrderDNN($p$) and outperforms PINN in high dimensions. As observed in Table \ref{tab:table6}, the amount of parameters required in HOrderDNN($p=9$) grows exponentially from $2.10E+12$ to $2.02E+52$ with the dimension $d$ increasing from $10$ to $50$, making them intractable in practical computations. In contrast, K-HOrderDNNs($p$) remain computationally feasible when $d=10, 20, 50$. Furthermore, the relative errors of K-HOrderDNN($p=9$) are reduced by two orders of magnitude with respect to PINN when $d=10$, which can be seen in Table \ref{tab:table9} for Test Problem 1 and Test Problem 3, and Table \ref{tab:table14}.
\par
The remainder of this paper is structured as follows. In the next section, we provide a brief overview of PINN and HOrderDNN, explain the reasons why HOrderDNNs suffer from CoD, review a specific version of KST, and introduce our novel K-HOrderDNN inspired by the KST. In Section 3, we analyze the approximation properties of our K-HOrderDNN and demonstrate that CoD is avoided for a certain class of functions. Section 4 presents a series of numerical examples to demonstrate the effectiveness of K-HOrderDNN. It is clearly observed that K-HOrderDNNs($p>1$) can achieve faster convergence and lower relative errors than PINN and HOrderDNNs($p$) for high frequency problems in high dimensions. The final section provides a brief summary and discussion.

\section{Methodology} 
\label{Method}
\subsection{Review of PINN and HOrderDNN}
In this section, We first review how to use the PINN to solve the boundary value problems of the following general PDE
\begin{equation}
\left\{
\begin{aligned} 
\mathcal{L}u(\boldsymbol{x} 
) &= f(\boldsymbol{x}), & \boldsymbol{x} &\in \Omega \\
\mathcal{B}u(\boldsymbol{x}) &= g(\boldsymbol{x}), & \boldsymbol{x} &\in \partial \Omega ,
\end{aligned}
\right.
\label{AGDO}
\end{equation}
where $\partial \Omega$ denotes the boundary of the domain $\Omega$, $\mathcal{L}$ is a differential operator such as $\nabla$, and $\mathcal{B}$ is a boundary operator that imposes conditions like Dirichlet or Neumann conditions. $f(\boldsymbol{x})$ and $g(\boldsymbol{x})$ refer to the source function and boundary conditions, respectively.
\par
A typical PINN employs a feed-forward neural network with $L-1$ hidden layers to approximate the exact solution, represented as:
\begin{equation}
u(\boldsymbol{x};\theta) = G_L \circ \sigma \circ G_{L-1} \circ \cdots \circ \sigma \circ G_{2} \circ \sigma  \circ G_1(\boldsymbol{x}),
\label{FNN}
\end{equation}
where $\sigma$ is an activation function (e.g., ReLU, tanh), $G_i(\boldsymbol{x}) = W_i\boldsymbol{x}+b_i$ represents the $i$-th hidden layer ($1 \le i \le L-1$) with weights $W_i$ and biases $b_i$. $\theta:= \left \{ W_i,b_i \right \}_{i=1}^L$ denotes the set of all trainable parameters.
To solve the problem (\ref{AGDO}), the PINN is trained by minimizing the following discrete physics-informed objective:
\begin{equation}
\theta^{\ast} = \arg\min_{\theta} L(\theta) = \arg\min_{\theta} \left\{ L_f(\theta) + \beta L_b(\theta) \right\},
\end{equation}
where $L_f(\theta)$ and $L_b(\theta)$, representing the discrete PDE loss and the boundary loss respectively, are defined as:
\begin{equation}
\left\{
\begin{aligned} 
L_f(\theta):=\frac{1}{N_f} \sum_{i=1}^{N_f} \left| \mathcal{L}u(\boldsymbol{x}_{f}^{i}; \theta) - f(\boldsymbol{x}_{f}^{i}) \right|^2,\\
L_b(\theta) := \frac{1}{N_b} \sum_{i=1}^{N_b} \left| \mathcal{B}u(\boldsymbol{x}_{b}^{i}; \theta) - g(\boldsymbol{x}_{b}^{i}) \right|^2,
\end{aligned}
\right.
\label{eq:eq15}
\end{equation}
and $\beta$ is the loss weight of boundary loss term. $\left \{  \boldsymbol{x}_{f}^{i} \right \}_{i=1}^{N_f} $ and $\left \{\boldsymbol{x}_{b}^{i} \right \}_{i=1}^{N_b}$ are sets of points randomly distributed in $\Omega$ and on $\partial \Omega$, respectively. 
\par
Although PINNs can efficiently solve a series of PDEs, they struggle with high frequency problems due to the 'F-Principle' or 'spectral bias'. To address this issue, \cite{chang2022high} introduces the HOrderDNN, whose architecture can be represented as: 
\begin{equation}
h_p(\boldsymbol{x};\theta) = G_L \circ \sigma \circ G_{L-1} \circ \cdots \circ \sigma \circ G_{2} \circ \sigma  \circ G_1 \circ T_p(\boldsymbol{x}),
\label{HOrderDNN}
\end{equation}
where $p$ denotes the order of HOrderDNN. As shown in Fig. \ref{fig: HOrderDNN}, the only difference to convential DNNs is the nonlinear transformation layer $T_p$, which maps the input  $x$ into a set of Lagrangian interpolation basis functions $\{  \Psi_i \}_{i=1}^{(p+1)^d}$ from the polynomial space $\mathcal{Q}(\mathbb{R}^d)$ with degree not exceeding $p$. In particular, basis functions $\{ \Psi_i \}_{i=1}^{(p+1)^d}$ are constructed by tensor products of one-dimensional functions $\{\psi_i\}_{i=1}^{p+1}$ , which can be expressed as:
\begin{equation}
\psi_j(x) = \prod_{i=1,i\ne j}^{p+1} \frac{x-x_i}{x_j-x_i},\quad j=1,2, \cdots,p+1,
\label{phi}
\end{equation}
where the interpolating nodes $x_i, i=1, 2, \cdots, p+1$ are chosen as Gauss-Lobatto-Legendre points in the calculation area.
\par
\begin{figure*}[htbp]
\centering
\includegraphics[width=1\textwidth]{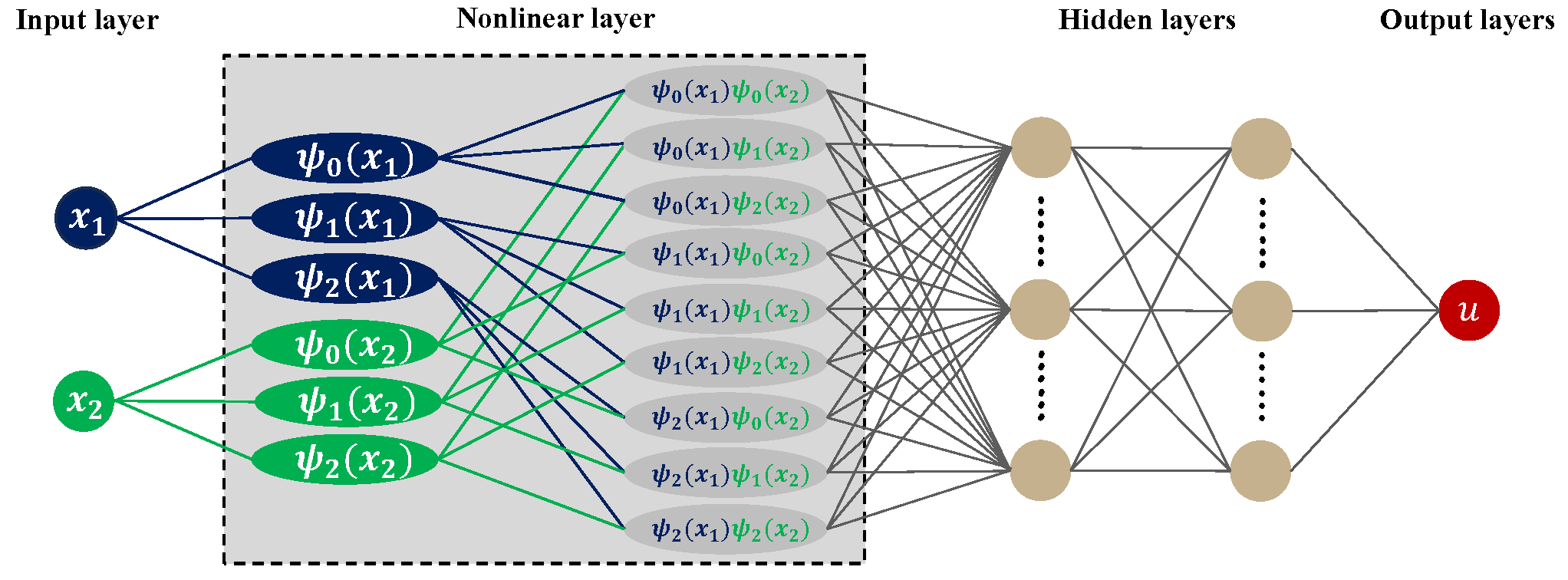}
\caption{Illustration of the architecture of HOrderDNN($p=2$) with $d=2$. The only difference to conventional DNNs is the nonlinear transformation layer followed by the input layer.}
\label{fig: HOrderDNN}
\end{figure*}
\par
As demonstrated in \cite{chang2022high,Chang_Li_Zou_Xiang}, HOrderDNNs outperform PINNs in terms of accuracy and speed for high frequency problems when $d\leq5$. However, they become impractical in higher dimensions due to the exponential growth of the required basis functions. 
To inherit the benefits of HOrderDNNs in low dimensions and overcome their limitations in higher dimensions, our K-HOrderDNNs combine the expressivity of HOrderDNNs in one-dimensional space and the Kolmogorov Superposition Theorem.

\subsection{Kolmogorov Superposition Theorem}
The Kolmogorov Superposition Theorem (KST for short) was originally proposed by Kolmogorov \cite{kolmogorov1957representation} in 1957 to answer Hilbert's 13th problem concerning the solution of algebraic equations. The original version of the KST states that for any continuous function $f:\left [ 0,1 \right ] ^{d} \to \mathbb{R}$, there exist $2d+1$ K-outer functions and $d(2d+1)$ K-inner functions to exactly represent the $d$-variate function $f$. Since then, more refined KSTs have been derived to reduce the count of univariate functions involved in the exact representation and to improve their smoothness. For example, G.G. Lorentz \cite{lorentz1962metric, lorentz1966approximation} simplified the construction to have only one K-outer function and $2d+1$ K-inner functions with $d$ irrational numbers. As pointed out in \cite{Doss_2017}, $2d+1$ is the minimum feasible number of inner functions required in the representation. Meanwhile, the continuity of the K-inner functions can be improved from the Lipschitz continuous in the original KST to be $\mathrm{Lip}_\alpha$ for any $\alpha \in (0,1)$. A more detailed introduction to the history of the KST can be found in \cite{Morris_2020}. Here, we focus on the KST version by G.G. Lorentz.
\par   
\begin{thm}[Kolmogorov Superposition Theorem] 
For any continuous function $f$ defined on $\left [ 0,1 \right ] ^{d}$, there exist irrational numbers $0< \lambda _{i} \le 1$ for $i=1,\,2,\,\cdots,\,d$, and strictly increasing $Lip \left ( \alpha  \right )$ inner functions $\phi_q$ (independent of $f$) with $\alpha = log_{10}2$ on $\left [ 0,1 \right ]$ for $q=0,\,1,\,\cdots,\,2d$, and with the presence of a continuous outer function $g(z),\,z \in \left [ 0,d \right ]$ such that the following identity holds:
\begin{equation}
f (x_{ 1},\cdots ,x _{d} ) = \sum_{q=0}^{2d} g\left ( \sum_{i=1}^{d} \lambda _{i} \phi _{q} \left ( x_{i} \right ) \right ).
\label{KST}
\end{equation}
\label{thm1}
\end{thm}
\par
For clarify, a function $f : [a, b] \rightarrow \mathbb{R}$ is said to belong to the class $\mathrm{Lip}_M(\alpha)$, for some constant $M \geq 0$ and $0 < \alpha \leq 1$, if 
\begin{equation*}
	|f(x) - f(y)| \leq M|x - y|^\alpha,\quad \forall\,x,\,y \in [a, b].
\end{equation*}
The class $\mathrm{Lip}(\alpha)$ consists of all functions that satisfy this condition for some $M$.
We remark that $\phi_q,\,q=0,\,1,\,\cdots,2d$ are called the K-inner functions, which are independent of the represented function $f$, and $g$ is called the K-outer function depending on $f$. We remark that achieving exact representations via KST is impractical in real-world computations, as the construction of univariate functions is an infinite process \cite{Sprecher_1997, Braun_Griebel_2009, Sprecher_Draghici_2002}. Hence, approximate versions of KST are extensively explored where K-inner and/or K-outer functions are approximated with simpler and smoother functions, such as splines function \cite{IGELNIK_PARIKH_2004,lai2022kolmogorov} and neural networks \cite{Kurkova_1991, Kurkova_1992, Schmidt-Hieber_2020, Montanelli_Yang_2019}. 

\subsection{The architecture of K-HOrderDNN}
Following a similar idea as in the approximate versions of KST, we utilize specially designed neural networks to approximate the K-inner and K-outer functions in our K-HOrderDNN.
\par
For clarity, we first rephrase the representation \eqref{KST} into a vectorized form. For any $x\in [0,1]$ and $\boldsymbol{z}=(z_0,z_1,\cdot\cdot\cdot,z_{2d})\in[0,d]^{2d+1}$, we define a univariate vector-valued function $\mathbf{\Phi}(x)$ and a multivariate vector-valued function $\mathbf{G}(\boldsymbol{z})$ as follows:
\begin{equation*}
\mathbf{\Phi}(x)=(\phi_0(x),\phi_1(x),\cdot\cdot\cdot,\phi_{2d}(x)),\quad \mathbf{G}(\boldsymbol{z})=(g(z_0),g(z_1),\cdot\cdot\cdot,g(z_{2d}))^T.
\end{equation*}
\par
Equation \eqref{KST} can be rewritten as 
\begin{equation}
f(x_1,\cdot\cdot\cdot,x_d)=V\cdot \mathbf{G}(\Lambda \cdot [\mathbf{\Phi}(x_1),\mathbf{\Phi}(x_2),\cdot\cdot\cdot,\mathbf{\Phi}(x_d)]^T).
\label{KST_vec}
\end{equation}
where $\Lambda=[\lambda_1 \mathbf{I}_{2d+1},\lambda_2 \mathbf{I}_{2d+1},\cdot\cdot\cdot, \lambda_d \mathbf{I}_{2d+1}]$ with $\mathbf{I}_{2d+1}$ being a $(2d+1)\times(2d+1)$ identity matrix, and $V=[1,1,\cdot\cdot\cdot,1]$ being an all-one row vector of length $2d+1$.
Motivated by the vectorized representation \eqref{KST_vec}, Our K-HOrderDNN is defined as 
\begin{equation}
k_p(x_1,\cdot\cdot\cdot,x_d) = G_2 \circ  \mathbf{G}_{NN}\circ\sigma\circ G_1([h_p(x_1),h_p(x_2),\cdot\cdot\cdot,h_p(x_d)]^T).
\label{K-HOrderDNN}
\end{equation}
Here, $h_{p}$ denotes a high order multi-layer subnetwork defined in \eqref{HOrderDNN} with one input neuron and $2d+1$ output neurons, aimed at imitating the univariate vector-valued inner function $\mathbf{\Phi}$. $\mathbf{G}_{NN}$ represents a fully connected subnetwork introduced to simulate the outer function $\mathbf{G}$,while $G_1,\, G_2,$ and $\sigma$ are defined in \eqref{FNN}.  
An example of the architecture of K-HOrderDNN($p=2$) when $d=2$ is shown in Fig. \ref{fig:2}. Compared to HOrderDNN depicted in Fig. \ref{fig: HOrderDNN}, K-HOrderDNN also incorporates the "high order" idea by introducing a non-linear transformation layer following the input layer, but the specific construction of the non-linear transformation layer is obviously different. Instead of inserting $(p+1)^d$ tensor-product basis functions from the multivariate polynomial space $\mathcal{Q}_p (\mathbb{R}^d)$, K-HOrderDNN($p$) only introduces $p+1$ basis functions from $\mathcal{Q}_p (\mathbb{R})$, followed by a parameter sharing subnetwork, which effectively avoids the explosion in parameter counts encountered in HOrderDNN when implemented in high dimensions.
\par
\begin{figure}[ht]
\centering
\includegraphics[width=\textwidth]{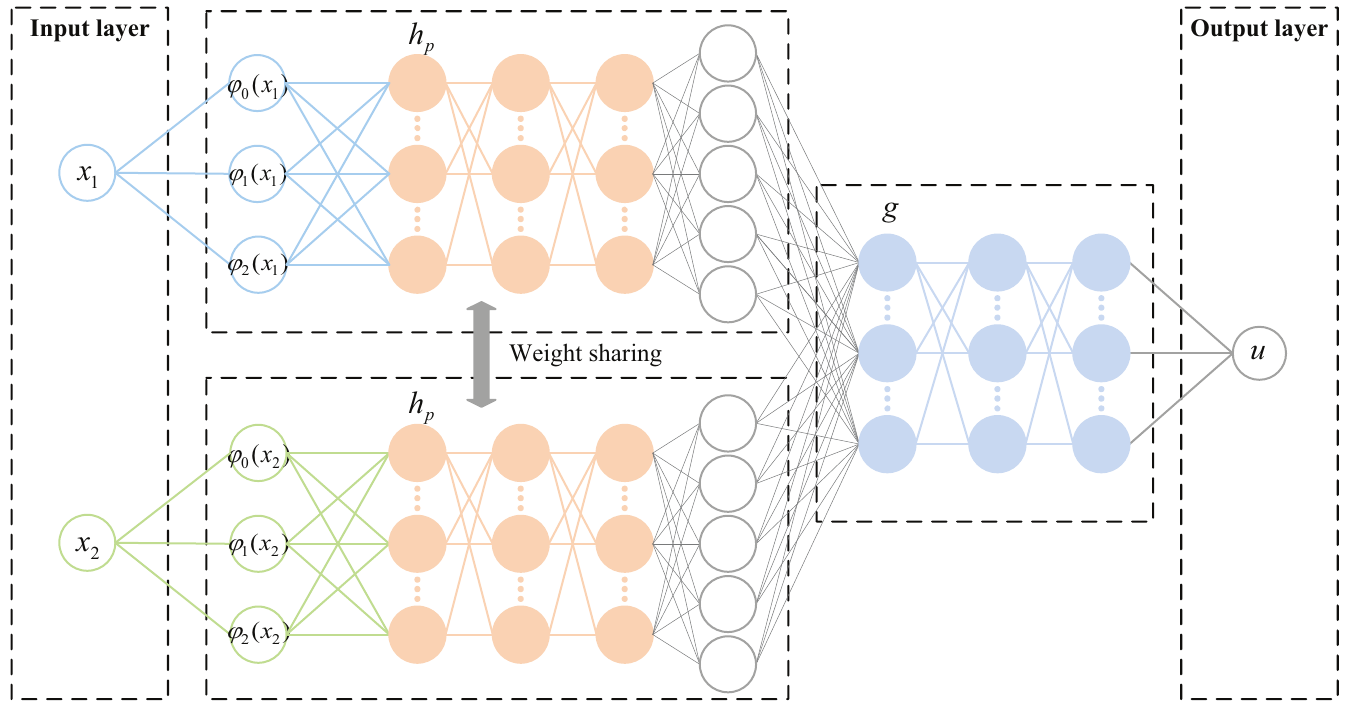}
\caption{Schematic illustration of K-HOrderDNN($p=2$) for $d=2$. Compared to HOrderDNN($p=2$) decipted in Fig. \ref{fig: HOrderDNN}, tensor product basis functions are avoided to prevent parameter explosion. Note that the subnetwork $g\,=\,\mathbf{G}_{NN}\circ\sigma\circ G_1$. The width and depth of the subnetworks $h_p$ and $g$ are provided in Table \ref{table:table1}, and specific values can be found in the detailed examples.}
\label{fig:2}
\end{figure}
\section{Theoretical analysis}
In this section, we will investigate the approximation capabilities of K-HOrderDNNs under the Rectified Linear Unit (ReLU) activation function and reveal that CoD is avoided when approximating a dense subset of continuous functions.
\par
Throughout the rest of this section, we will use $\sigma$ to denote the ReLU function. For any continuous function $f$, we utilize polynomials to approximate associated K-inner functions $\phi_q$ and a linear spline to approximate the associated K-outer function $g$. Specifically, let $L_q(x)$ be polynomials of degree $p$ which can be expressed as
\begin{equation*}
		L_q(x) := \sum_{j=1}^{p+1}a_{j,q}\psi_j(x),\quad q=0,1,\cdots,2d,
\end{equation*}
where $a_{j, q} \in \mathbb{R}$ and $\psi_j(x)$ refers to the one-dimensional basis functions defined previously in equation $\left( \ref{phi} \right)$.
Note that the K-inner functions $\phi_q$ are univariate functions mapping from $[0,1]$ to $[0,1]$. To restrict the values of $L_q(x)$, we introduce 
\begin{equation}
\widetilde{L}_{q}(x) := \min\left\{\max\left\{L_{q}(x), 0\right\}, 1\right\}\approx \phi _{q} (x),\quad q=0,1,\cdots,2d.
\label{L_q}
\end{equation}
To approximate the K-outer function $g$, we follow \cite{lai2022kolmogorov} and adopt a linear spline $S_g(z)$ which can be written in terms of a linear combination of ReLU functions as follows:
\begin{equation}
S_g(z) := \sum_{k=1}^{n}w_{k}\sigma  ( z-z_k )  \approx g(z).
\label{S_g}
\end{equation} 
\par
Upon substituting equations \eqref{L_q} and \eqref{S_g} into equation \eqref{KST}, we obtain an approximation of the original function $f(x_1, \ldots, x_d)$ as follows:
\begin{equation*}
f(x_1,x_2,\cdots,x_d)\approx \sum_{q=0}^{2d} S_g \left(\sum_{i=1}^{d} \lambda_{i} \widetilde{L}_{q}(x_i)\right).
\end{equation*}
We remark that the above approximation can be interpreted as a special class of K-HOrderDNN. Indeed, 
by representing the minimum and maximum operations as linear combinations of ReLU functions, equation \eqref{L_q} can be rewritten as 
\begin{equation}
\widetilde{L}_{q}(x) =1-\sigma(1-\sigma(\sum_{j=1}^{p+1}a_{j,q}\psi_j(x)))\approx \phi _{q} (x),\quad q=0,1,\cdots,2d,
\label{L_q_2}
\end{equation}
which means the subnetwork $h_p$ in K-HOrderDNN is chosen as a high order multi-layer neural network with one nonlinear transformation layer and three hidden layers with $2d+1$ neurons in each hidden layer.
Furthermore, equation \eqref{S_g} means the subnetwork $G_{NN}$ is chosen as a fully connected subnetwork with one hidden layer containing $n$ neurons.
For convience, we use $\mathcal{K}_{p,n}$ to denote this special class of K-HOrderDNN, i.e. 
\begin{equation*}
\mathcal{K}_{p,n}^{\sigma}:= \left \{  \sum_{q=0}^{2d} S_g (\sum_{i=1}^{d} \lambda_{i} \widetilde{L}_{q}(x_i)):
\begin{aligned}
&\text{$S_g$, $\widetilde{L}_{q}$ are defined in \eqref{S_g} \eqref{L_q_2}, respectively.} \\
&w_k,a_{j,q}\in \mathbb{R},\lambda_i\in (0,1],z_k\in[0,d].
\end{aligned}
\right\}
\label{eq:8}
\end{equation*} 
The parameters in $\mathcal{K}_{p,n}$ consist of $w_k,z_k$ with $k=1,\cdots,n$, $\lambda_i$ with $i=1,\cdots,d$ and $a_{j,q}$ with $j=1,\cdots,p+1,q=0,\cdots,2d$. Therefore, the total number of parameters equals to $2n+d+(p+1)(2d+1)$.
\par
Before presenting the main results of this section, let us revisit two auxiliary lemmas crucial to the proof.

\par
\begin{lem}[Jackson Theorem\cite{wang2019numerical}]
For any $ f(x) \in \mathrm{Lip}_M(\alpha)$ with $0<\alpha\leq 1$, there exists a polynomial $q(x) \in \mathcal{Q}_p (\mathbb{R})$ such that 
\begin{equation}
\max_{x\in [a,b]}\left| f(x) - q(x) \right| \le  \frac{K M}{p^{\alpha}},
\end{equation}
where \( K \) is a positive constant only dependent on \( a \) and \( b \).
\label{Jackson}
\end{lem}
\par

\begin{lem}[\cite{lai2022kolmogorov}]
Suppose that $g$ is  Lipschitz continuous over $[a, b]$ with Lipschitz constant $C_g$. For any $n\geq1$, there exists a partition 
$\bigtriangleup$ 
with $n$ interior knots such that 
\begin{equation*}
\inf_{s \in S_{1}^{0} \left ( \bigtriangleup  \right )}  \left \| g-s \right \|_{ C([a,b])} \le \frac{C_g\left ( b-a \right ) }{2(n+1)}, 
\end{equation*}
where $S_{1}^{0} \left ( \bigtriangleup  \right )$ denotes the space of all continuous linear splines over the partition $ \bigtriangleup = \{a = z_{0}< z_{1} < \cdots < z_{n} = b\}$.	
\label{Lemma2}
\end{lem}
We first state the approximation result for the class of K-Lipschitz continuous functions \cite{lai2022kolmogorov}, which consists of all Lipschitz continuous functions whose K-outer functions are also Lipschitz continuous, i.e.
\begin{equation*}
KL = \left\{ f\in C([ 0,1]^d): \text{K-outer function } g \ \text{is Lipschitz continuous} \right\}.
\end{equation*}
We remark that the new function class $KL$ is dense in $C([0,1]^d)$ \cite{lai2022kolmogorov}.
\begin{thm}
Suppose $f\in C([0,1]^d)$ belongs to $KL$. Let $C_g$ denote the Lipschitz constant of the associated K-outer function $g$ and $C_\phi$ be the common Lipschitz constant for the assciated K-inner functions $\phi_q, q=0,1,\cdots,2d$. Then we have
\begin{equation}
\inf_{s\in \mathcal{K}_{p,n}^{\sigma}}\left \| f-s \right \| _{ C([ 0,1]^d)   } \le C_{g}d(2d+1) (\frac{1}{n}+\frac{2KC_\phi}{p^{\alpha}}),
\end{equation}
with $\alpha = \log_{10}2$.
\label{thm2}
\end{thm}
\par
\begin{proof}
Let $\phi_q, q=0,1,\cdots,2d$ be the K-inner functions associated with $f$. By Theorem \ref{thm1}, $\phi_q \in Lip_{C_\phi}(\alpha)$ with $C_\phi$ being their common Lipschitz constant and $\alpha = \log_{10}2$. According to Lemma \ref{Jackson}, there exist  polynomials $L_q \in \mathcal{Q}_p (\mathbb{R})$ such that 
\begin{equation}
\max_{x\in [0,1]}\left|\phi_q(x)-L_q(x)\right|\le  \frac{KC_\phi}{p^{\alpha}},\quad q=0,1,\cdots ,2d.
\end{equation}
It is straightforword to confirm that $|\phi_{q}(x)-\widetilde{L}_q(x)| \le |\phi_{q}(x)-L_q(x)|$ by the definition of $\widetilde{L}_{q}(x)$.  
Hence, we have
\begin{equation}
\max_{x\in [0,1]} \left|\phi_q(x)-\widetilde{L}_q(x)\right|\le \frac{KC_\phi}{p^{\alpha}},\quad  q=0,1,\cdots ,2d.
\label{eq2}
\end{equation}
Let $g$ be the K-outer function  associated with $f$, which is Lipschtiz continuous with a Lipschitz constant $C_g$. By Lemma \ref{Lemma2}, there is a linear spline $S_g \in S_{1}^{0} \left ( \bigtriangleup  \right )$ over a partition $ \bigtriangleup  =\left \{ 0= z_{0}< z_{1} < \cdots < z_{n}=d\right \} $ such that
\begin{equation*}
\left \| g -S_{g}  \right \|_{C([0,d])} \le \frac{C_{g}d}{2(n+1)} \le \frac{C_{g}d}{2n}.
\end{equation*}
Following the equation \eqref{KST} in Theorem \ref{thm1}, we have 
\begin{equation*}
\left | f(x)-\sum_{q=0}^{2d} S_g\left ( \sum_{i=1}^{d} \lambda _{i}  \phi _{q}\left ( x_{i}  \right )  \right )  \right | \le \sum_{q=0}^{2d}\left |  g\left ( \sum_{i=1}^{d} \lambda _{i}  \phi _{q}\left ( x_{i}  \right )  \right )-S_g\left ( \sum_{i=1}^{d} \lambda _{i}  \phi _{q}\left ( x_{i}  \right )  \right )\right | \le \frac{C_{g}(2d+1)d}{2n}. 
\end{equation*}
We remark that when $g$ is Lipschitz continuous, so is its linear interpolatory spline $S_g$. Specifically, we have $\left | S_g(x)-S_g(y) \right | \le 2C_{g} \left | x-y \right |$.
Thus, by using equation \eqref{eq2}, we obtain 
\begin{equation*}
|S_g\left ( \sum_{i=1}^{d} \lambda _{i}  \phi _{q}\left ( x_{i}  \right )  \right )-S_g\left ( \sum_{i=1}^{d} \lambda _{i}  \widetilde{L}_{q}\left ( x_{i}  \right )  \right )| \le 2C_{g} \sum_{i=1}^{d}\lambda _{i}\left | \phi _{q}\left ( x_{i}  \right ) -\widetilde{L}_{q}\left ( x_{i}  \right )   \right | \le  \frac{2C_{g} d K C_\phi}{p^{\alpha}}.
\end{equation*}
Combining above two inequalities we have
\begin{equation*}
\left | f(x)-\sum_{q=0}^{2d} S_g\left ( \sum_{i=1}^{d} \lambda _{i}  \widetilde{L}_{q}\left ( x_{i}  \right )  \right )  \right | \le \frac{C_{g}(2d+1)d}{2n}+\frac{2C_{g} K C_\phi(2d+1)d}{p^{\alpha}}\le C_{g}(2d+1) d(\frac{1}{n}+\frac{2KC_\phi}{p^{\alpha}}).
\end{equation*}
This completes the proof.
\end{proof}
Theorem \ref{thm2} implies that for a dense subclass of continuous functions, only $2n+d+(p+1)(2d+1)$ parameters are required to achieve the approximation rate $O(n^{-1}+p^{-\alpha}) 
$, with the approximation constant increasing at a quadratic rate as the dimension increases. Consequently, K-HOrderDNNs break the curse of dimension for functions in this dense subclass, making them tractable even in high dimensions. 
Futhermore, as the K-outer function $g$ may not be Lipschitz continuous, we further extend Theorem \ref{thm2} to general continuous functions using the standard modulus of continuity to characterize the smoothness of $g$. For any $g\in C([a,b])$, the modulus of continuity of $g$ is defined by 
\begin{equation*}
	\omega(g, \delta) = \sup \left\{ |g(x) - g(y)| : x, y \in [a, b], |x-y| \leq \delta \right\},
\end{equation*}
for any $\delta>0$. 
Then we conclude the following theorem, with its proof deferred to Appendix A.
\par
\begin{thm}
For any continuous function $f \in C([ 0,1]^d) $, let $g$ be the K-outer function associated with $f$ and $C_\phi$ be the common Lipschitz constant for the assciated K-inner functions $\phi_q, q=0,1,\cdots,2d$. Then we have
\begin{equation}
\inf_{s\in \mathcal{K}_{p,n}^{\sigma}}\left \| f-s \right \| _{ C([ 0,1]^d)   }
\le d(2d+1)\left ( 3\omega \left ( g, \frac{1}{n} \right ) +(K C_\phi +1)\omega(g,\frac{1}{p^{\alpha}}) \right ) .
\end{equation}
\label{Theorem4}
\end{thm}
\par
Theorem \ref{thm2} and Theorem \ref{Theorem4} are derived under the assumption of using the ReLU activation function. Next, we extend the approximation result to the Tanh activation function $ \sigma_1$ for a dense subclass of continuous functions, with the proof provided in Appendix E.	
For convience, we use $\mathcal{K}_{p,n}^{\sigma_1}$ to denote the special class of K-HOrderDNN with tanh activation function $\sigma_1$, i.e. 
\begin{equation*}
	\mathcal{K}_{p,n}^{\sigma_1} := \left\{ \sum_{q=0}^{2d} \hat{g}^N \left( \sum_{i=1}^{d} \lambda_{i} \bar{L}_{q}(x_i) \right) :
	\begin{array}{l} \bar{L}_{q} \text{ is defined in \eqref{req:1}, } a_{j,q} \in \mathbb{R}, \, \lambda_i \in (0,1], \hat{g}^N \text{ is a tanh neural}\\
		\text{network with two hidden layers, where the layer widths } \\
		\text{are at most } (N - 1) \text{ and } 6N, \text{ as presented in Lemma \ref{lemm3:tanh},}\\
	\end{array}
	\right\}
\end{equation*}
where for $ N \in \mathbb{N} $ with $ N > 5 $.
The parameters of $\mathcal{K}_{p,n}^{\sigma_1}$ consist of $a_{j,q}$  with $j=1,\cdots,p+1,q=0,\cdots,2d$, along with the parameters of the fully connected neural network $\hat{g}^N(t; \theta)$. The input dimension of $\hat{g}^N(t; \theta)$ is $2d+1$, and the output is 1. The total number of parameters of $\hat{g}^N(t; \theta)$ is given by $(2d+6N+14)(N-1)+13$. Therefore, the overall number of parameters in $\mathcal{K}_{p,n}^{\sigma_1}$ is $(2d+6N+14)(N-1)+13+(p+1)(2d+1)$.
\begin{thm}
	Suppose $f\in C([0,1]^d)$ belongs to $KL$. Let $C_g$ denote the Lipschitz constant of the associated K-outer function $g$ and $C_\phi$ be the common Lipschitz constant for the assciated K-inner functions $\phi_q, q=0,1,\cdots,2d$. Then, for $ N \in \mathbb{N} $ with $ N > 5 $, we have
	\begin{equation}
		\inf_{s\in \mathcal{K}_{p,n}^{\sigma_1}}\left \| f-s \right \| _{ C([ 0,1]^d)   } \le C_g(2d+1)d (\frac{21}{N}+\frac{3KC_\phi}{p^{\alpha}}),
	\end{equation}
	with $\alpha = \log_{10}2$.
	\label{thm_tanh}
\end{thm}
Theorem \ref{thm_tanh} implies that for a dense subclass of continuous functions, only $(2d+6N+14)(N-1)+13+(p+1)(2d+1)$ parameters are required to achieve the approximation rate $O(N^{-1}+p^{-\alpha}) 
$, with the approximation constant increasing at a quadratic rate as the dimension increases. Consequently, K-HOrderDNNs with tanh activation also break the curse of dimension for functions in this dense subclass, making them tractable even in high dimensions.
\section{Numerical results} 
In this section, we conduct various numerical experiments, including function fitting and PDEs solving, to evaluate the performance of K-HOrderDNNs. We incrementally escalate problem complexity, starting with a simple low dimensional, high frequency problem and gradually moving to more a complex high dimensional, high frequency problem. Throughout all experiments, we provide comprehensive comparisons against PINN and HOrderDNNs. Our results show that K-HOrderDNNs outperform both HOrderDNNs and PINN methods in all cases. 
\subsection{Experiments setup}
The hyperparameter settings in our experiments are as follows. To ensure a fair comparison, the depths of \text{PINN} and \text{HOrderDNN} match the depth of \text{K-HOrderDNN}, and the widths of \text{PINN} and \text{HOrderDNN} align with the maximum number of neurons in \text{K-HOrderDNN}. Unless otherwise specified, we use the ReLU activation function for function fitting problems. Considering the need to compute the second derivative of the neural network, and the fact that the second derivative of the ReLU function is zero almost everywhere\cite{maczuga2023influence}, which hinders learning, we opt to use the Tanh activation function to solve the equations. We initialize all model weights using the \text{Xavier} initialization and train them using the \text{Adam} optimizer for $50,000$ epochs. The initial learning rate is set to $4e-3$, and then a decay of $0.9$ is applied after every $1000$ epoch. The accuracy between the true solution $u^*$ and its numerical approximation $u(\mathbf{x};\theta)$ is measured by the relative $L_2$ error, i.e., $\text{REL} = \frac{\|u(\mathbf{x};\theta ) - u^*\|_2}{\|u^*\|_2}$. To assess this $L_2$ error, points are sampled in advance. In the $2D$ case, they comprise $m$ equidistant grid points $(x_1, x_2)$, where $m = 100 \times 100$. In the $d$-dimensional case, we extend each 2D grid point into higher dimensions by adding extra coordinates that are randomly drawn values of $(x_3, \ldots, x_d)$ within domain $\Omega$. Table \ref{table:table1} presents the necessary parameter notations for quick reference. Readers can find the values of these parameters in the experiment description below.
\par
\begin{table}[h]
\centering
\caption{Summary of model parameters}
\smallskip
\setlength{\tabcolsep}{3pt}
\resizebox{\textwidth}{!}{
\begin{tabular}{||ll||}
\hline
Notation & Stands for \\ \hline 
$d$ & Dimension of $\Omega \subset \mathbb{R}^d$ \\
$p$ & The order of HOrderDNN or K-HOrderDNN  \\
$L$ & The depth of HOrderDNN or PINN \\
$W$ &  The width of HOrderDNN or PINN  \\
$hd$ & The depth of $h_p$ used in K-HOrderDNN\\
$hw$ & The width of $h_p$ used in K-HOrderDNN\\
$gd$ & The depth of $g$ used in K-HOrderDNN \\
$gw$ & The width of $g$ used in K-HOrderDNN \\
$hd+gd+1$ & The depth of K-HOrderDNN \\
$max(hw,gw)$ &  The width of K-HOrderDNN  \\
$N_b$ & In each training step, the number of boundary points sampled on lines or surfaces.\\
$N_f$ & In each training step, the number of interior points sampled in the region $\Omega$. \\
\hline
\end{tabular}
}
\label{table:table1}
\end{table}

\subsection{Effectiveness of K-HOrderDNN}
In this section, we demonstrate the capability of K-HOrderDNN to capture high frequency components of target functions across four aspects: parameter settings, convergence process, frequency perspective, and application in high-dimensional problems.

\subsubsection{Comparison on various parameter settings}
\label{f1}
\par
\begin{figure*}[htbp]
	\centering
	\includegraphics[width=\textwidth]{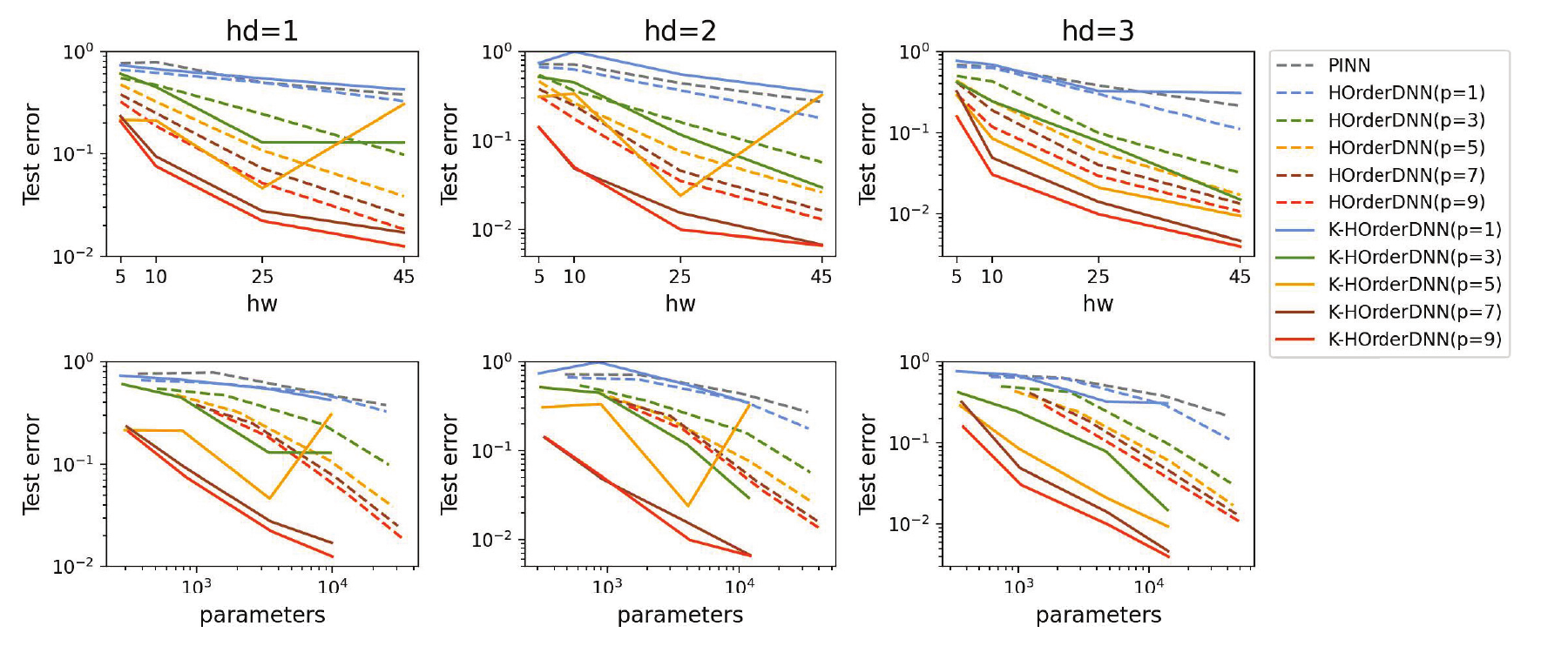}
	\caption{The change in relative $L_2$ errors with respect to hw (the first row) and the number of trainable parameters (the second row) for PINN, HOrderDNNs, and K-HOrderDNNs on problem (\ref{equation2}) when $d$=2. Each column corresponds to a different setting of $hd$, specifically $hd$ = 1, 2, and 3.}
	\label{fig:3}
\end{figure*}
We consider the following function fitting problem in domain $\Omega=[0,1] \times[0,1]$,
\begin{equation}
f(x_1, x_2) = g(x_1) \cdot g(x_2), \quad\,x_1,\,x_2\,\in [0,1],
\label{equation2}
\end{equation}
where $g(x_i) = \sum_{j=0}^{5} \sin(2^j \pi x_i)$. We select the discrete mean squared error (MSE) as the loss function, defined by 
\begin{equation}
L(\theta) = \frac{1}{N_f} \sum_{i=1}^{N_f} \left( f\left(x^{(i)}\right) - u_p\left(x^{(i)}; \theta\right) \right)^2,
\end{equation}
where $u_p(\cdot; \theta)$ is the approximated solution. In each training epoch, we resample $N_f=16,000$ random training points $\{x^{(i)}\}$.   
\par
The hyperparameter settings for K-HOrderDNNs are specified as follows: $hw$= 5, 10, 25, 45; $hd$ = 1, 2, 3; $gd$ = 2, with the $gw$ setting at twice that of $hw$, and $p=1, 3, 5, 7, 9$. For HOrderDNNs and PINN, we set $L=hd+gd+1$ and keep $W$ aligned with $gw$.
\par
The effect of $p$ and dynamic changes of relative errors with respect to $hw$ under different $hd$ conditions are shown in Fig. \ref{fig:3}. As observed, for almost all cases, if $hd$, $gd$, and $gw$ are fixed, larger $hw$ yields smaller relative errors, and for the fixed $hw$, $gd$, $gw$, relative errors decrease as $hd$ increases. We also note that K-HOrderDNNs($p>1$) outperform PINN and HOrderDNNs($p>1$). In particular, when $hd = 3$, $hw = 45$, K-HOrderDNN($p=9$) achieves a relative error of 3.97E$-$3, which is at least two orders of magnitude smaller than PINN (2.13E$-$1), and at least one order of magnitude smaller than HOrderDNN($p=9$) (1.06E$-$2), as shown in Appendix B, Table \ref{tab:B1}. Like HOrderDNN($p$), as the order $p$ increases, K-HOrderDNN($p$) also experiences a significant decrease in relative errors. This observation aligns well with Theorem \ref{thm2}. In addition, we present the changes in relative errors concerning the number of trainable parameters under different $hd$ in Fig. \ref{fig:3}. As we observe, when the number of parameters is held constant, an increase in the order $p$ leads to a reduction in relative error for both HOrderDNN and K-HOrderDNN. Furthermore, with the order $p$ fixed, K-HOrderDNN($p$) exhibits lower errors than HOrderDNN($p$). To achieve the same relative error, K-HOrderDNN($p$) requires significantly fewer parameters than HOrderDNN($p$), and similar to HOrderDNN($p$), K-HOrderDNN($p$) with larger values of $p$ requires fewer parameters.

\par
We set $hw = 45$, $hd = 1$ to further investigate the effect of increasing $gw$ and $gd$ on error reduction. 
As we can observe, in almost all cases, when $gd$ remains constant, an increase in $gw$ leads to a reduction in relative errors, which agrees well with Theorem \ref{thm2}. Simultaneously, the relative error decreases as $gd$ increases, assuming $gw$ remains unchanged. Fig. \ref{fig:4} shows the change in relative errors along with the number of trainable parameters. In line with our previous observations, with a fixed number of trainable parameters, the error decreases as the order $p$ increases. Additionally, achieving the same relative error requires fewer parameters for K-HOrderDNN($p$) with a larger $p$. Detailed numerical results are in Appendix B, Table \ref{tab:B2}.
\par
\begin{figure}[ht]
\centering
\includegraphics[width=\textwidth]{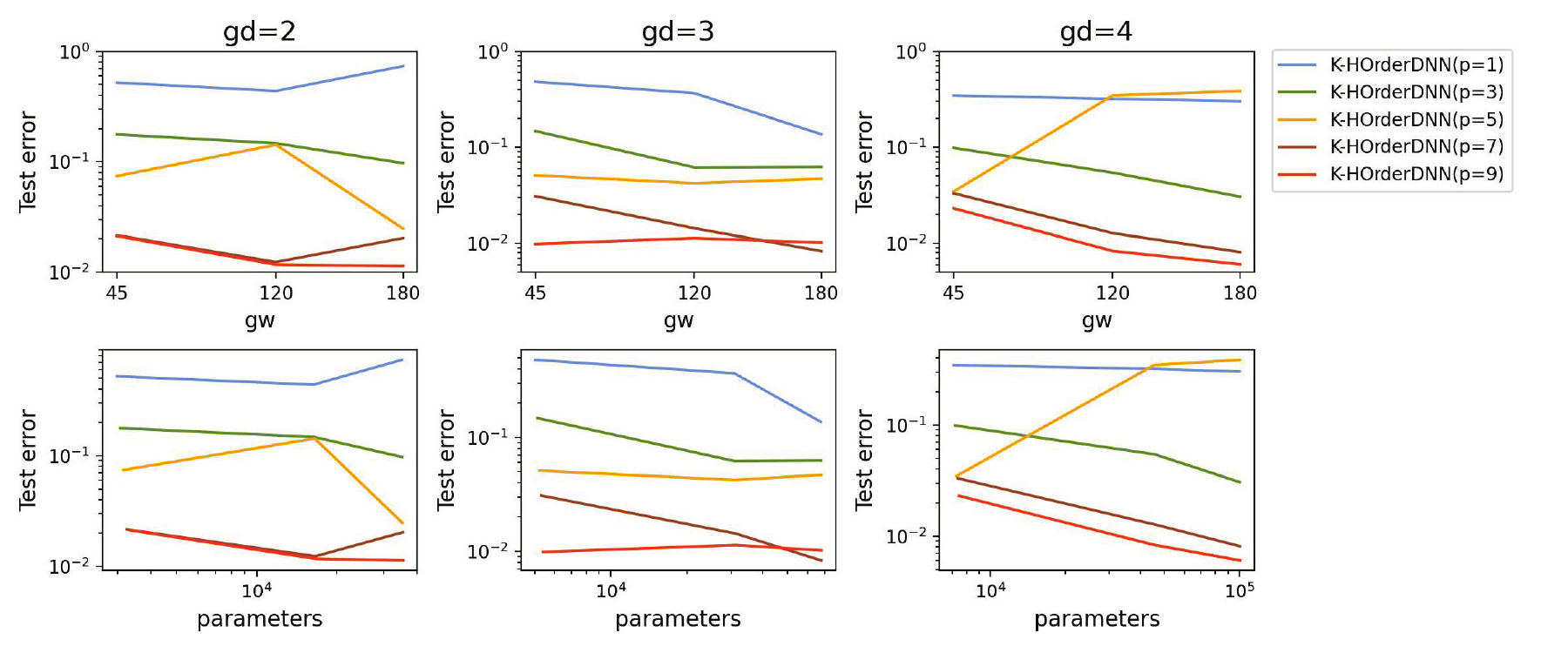}
\caption{The change in relative $L_2$ errors with respect to $gw$ (the first row) and the number of trainable parameters (the second row) for K-HOrderDNN($p$) on the problem (\ref{equation2}) when $d$=2. Each column corresponds to a different setting of $gd$, specifically $gd$ = 2, 3, and 4.}
\label{fig:4}
\end{figure}
\subsubsection{Comparison of convergence process}
\label{sec:sec1}
\par
\begin{figure}[H]
	\centering
	\includegraphics[width=\textwidth]{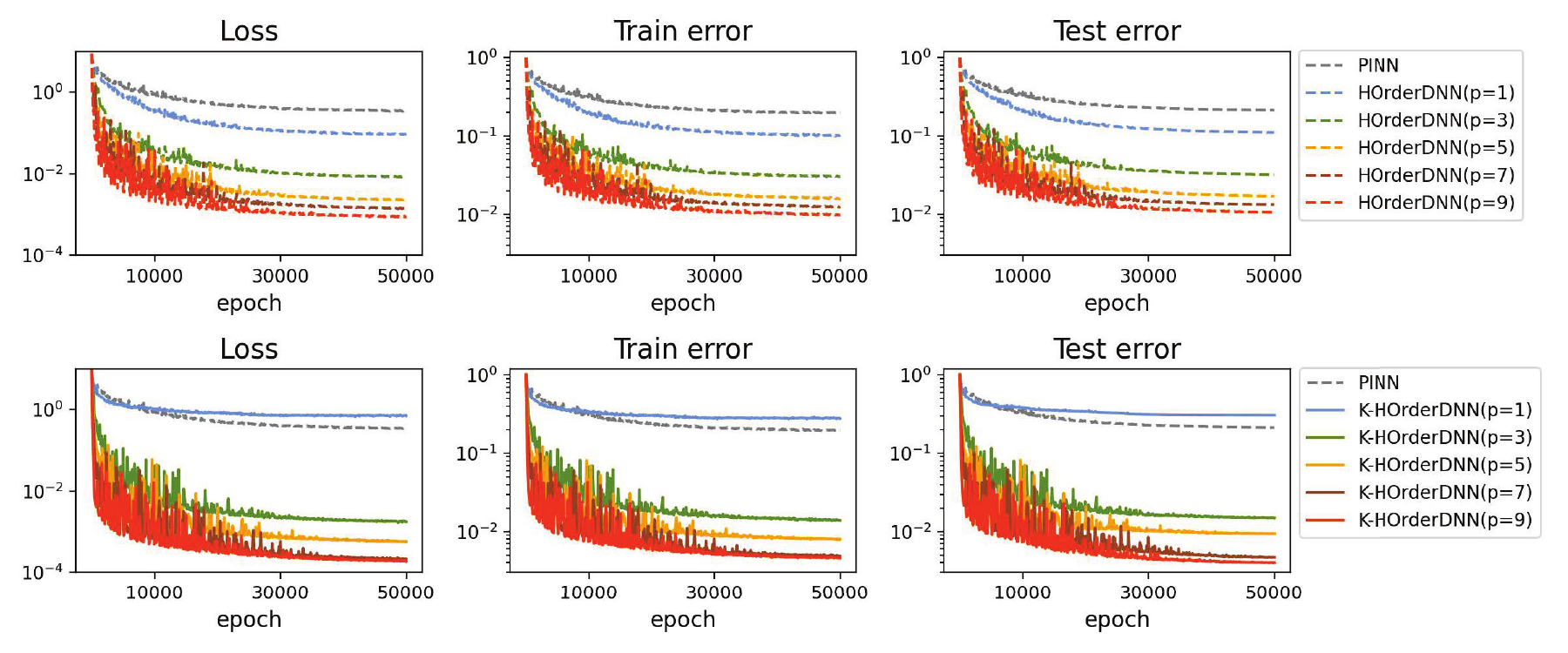}
	\caption{Convergence processes of HOrderDNNs (the first row), K-HOrderDNNs (the second row), and PINN under different order p on the problem (\ref{equation2}) when d=2. Here, $hd$=3, $hw$=45, $gd$=2, $gw$=90 for K-HOrderDNNs, and $L$=6, $W$=90 for HOrderDNNs and PINN.}
	\label{fig:5}
\end{figure} 
\par
Fig. \ref{fig:5} depicts the convergence processes of different models. This examination showcases results obtained by setting '$hd=3, hw=45, gd=2, gw=90$' for K-HOrderDNNs, and '$L=6, W=90$' for HOrderDNNs and PINN. It is evident that the proposed K-HOrderDNNs($p>1$) exhibit faster convergence compared to PINN and HOrderDNNs, achieving smaller errors and demonstrating its efficiency in fitting high-frequency components. 
Furthermore, like HOrderDNN($p$), the larger the value of $p$ in K-HOrderDNN($p$), the faster the convergence speed and the smaller the errors. Notably, similar results are observed for other cases.
\par
According to the absolute pointwise errors in Fig. \ref{fig:6}, we observe that the PINN, HOrderDNN($p=1$), and K-HOrderDNN($p=1$) exhibit higher errors, indicating their limited ability to capture local oscillations. In contrast, both HOrderDNNs and K-HOrderDNNs with the order $p>3$ demonstrate significantly lower errors, illustrating that they can capture oscillations at various scales. Additionally, for a fixed order $p$, K-HOrderDNNs($p>3$) always outperform HOrderDNNs, which indicates that the superior capability of K-HOrderDNNs in approximating the high-frequency components of the target functions.
\par
\begin{figure}[ht]
	\centering
	\includegraphics[width=\textwidth]{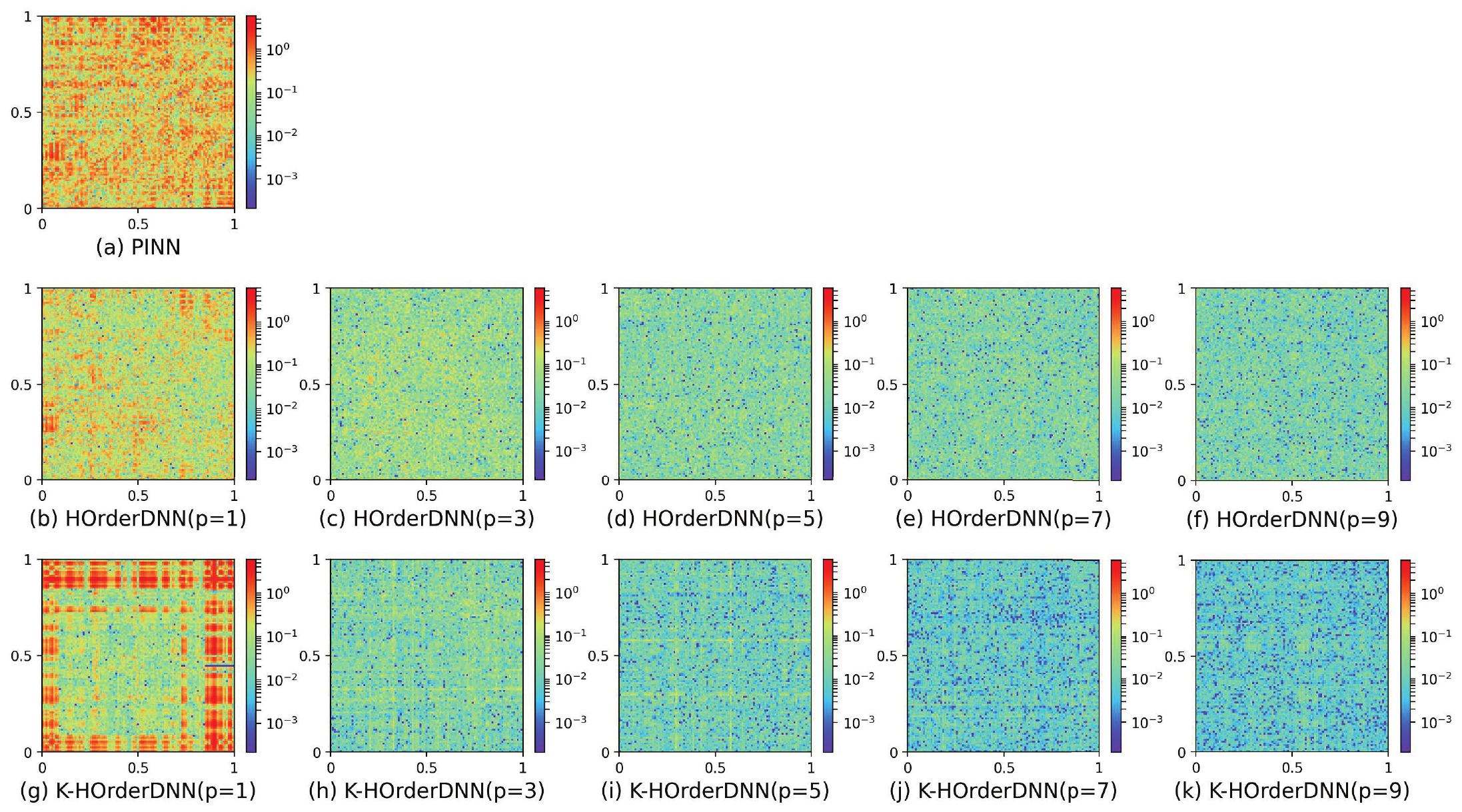}
	\caption{Absolute pointwise errors of PINN (the first row), HOrderDNNs (the second row), and K-HOrderDNNs (the third row) under different order $p$ on the problem (\ref{equation2}) when $d$=2. Here, $hd$=3, $hw$=45, $gd$=2, $gw$=90 for K-HOrderDNNs, and $L=$6, $W=$90 for HOrderDNNs and PINN.}
	\label{fig:6}
\end{figure}

\subsubsection{Comparison from frequency perspective}
\label{f3}
We further demonstrate the convergence behavior at different frequencies ($\gamma$ = 2, 4, 8, 16) during the training process of PINN, HOrderDNNs, and K-HOrderDNNs. Following Chang et al. \cite{chang2022high}, we perform the discrete Fourier transform (DFT) only along the first coordinate of the 2D testing data and then average the frequencies along the second coordinate. We compute the DFT of the target and approximating functions as
\begin{equation*}
\mathcal{F}[f] (\gamma)= \frac{1}{n} \sum_{j=0}^{n-1} f(x^{(j)}) \exp(-2 \pi i x^{(j)} \gamma),
\end{equation*}
and 
\begin{equation*}
	\mathcal{F}[u_p] (\gamma)= \frac{1}{n} \sum_{j=0}^{n-1} u_p(x^{(j)}; \theta) \exp(-2 \pi i x^{(j)} \gamma),
\end{equation*}
where $\gamma$ denotes frequency.
\par
\begin{figure}[ht]
	\centering
	\includegraphics[width=\textwidth]{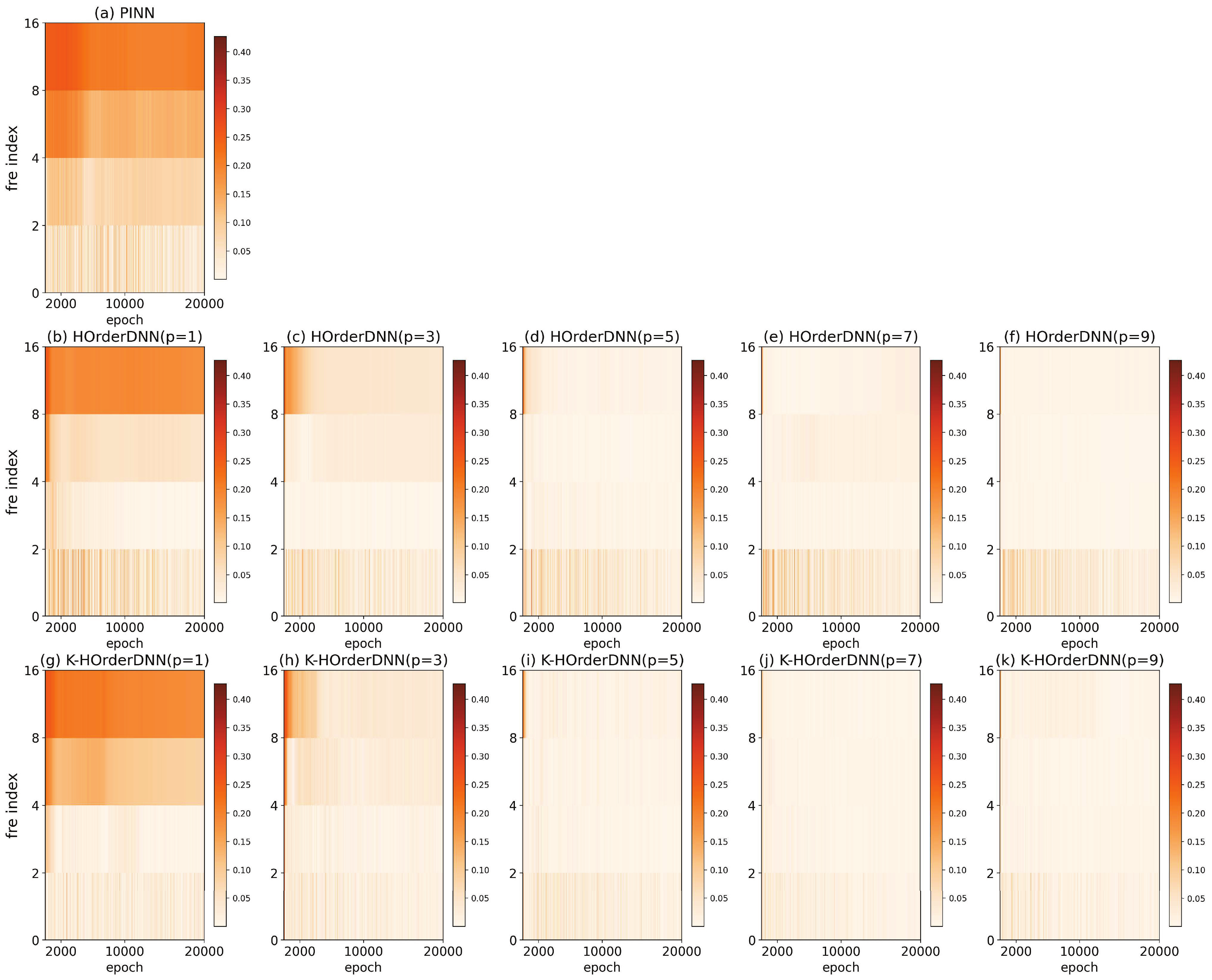}
	\caption{Frequency analysis at different epochs for PINN (the first row), HOrderDNNs (the second row), and K-HOrderDNNs (the third row) with $p$=1, 3, 5, 7, 9 on the problem(\ref{equation2}) when $d$=2. Here, $hd$=2, $hw$=5, $gd$=2, $gw$=10 for K-HOrderDNNs, and $L$=5, $W$=10 for HOrderDNNs and PINN.}
	\label{fig:7}
\end{figure}
\par
Fig. \ref{fig:7} illustrates the differences between $\mathcal{F}[f](\gamma)$ and $\mathcal{F}[u_p](\gamma)$ at the different frequencies ($\gamma$ = 2, 4, 8, 16) over epochs 2000, 10,000, and 20,000. It is evident that PINN, HOrderDNN($p=1$), and K-HOrderDNN($p=1$) rapidly learn low-frequency components while exhibiting slow learning speed for high-frequency components. Conversely, both HOrderDNNs($p > 3\allowbreak$) and K-HOrderDNNs($p>3$) effectively accelerate the learning of different frequency components almost simultaneously, with this effect significantly increasing as $p$ rises. Notably, K-HOrderDNNs with $p>3$ accelerate learning speed more effectively than HOrderNNs with $p>3$. Specifically, the K-HOrderDNN($p=9$) reduces the error at the different frequencies to below 0.05.
\subsubsection{Comparison on high-dimensional problems}
\label{f4}
We consider the following function fitting problem in domain $\Omega = [0,1]^d \subset \mathbb{R}^d$
\begin{equation}
\begin{aligned}
f(x_1, x_2, \ldots, x_d) = \sum_{i=1}^{d-2} g(x_i) \cdot g(x_{i+1}) \cdot g(x_{i+2}), \quad x_i \in [0, 1],
\end{aligned}
\label{equation3}
\end{equation}
where $g(x_i)$ is defined in Equation \eqref{equation2}.

We study the effectiveness of K-HOrderDNN($p$) with the dimension d = 10, 20, and 50. Table \ref{tab:table5} summarises the parameter settings for different methods. For simplicity, in each training epoch, we resample $N_f = 35,000$ random training points, and the number of $N_f$ does not increase with dimension.

Table \ref{tab:table6} lists the relative $L_2$ errors and the number of parameters of PINN, HOrderDNNs, and K-HOrderDNNs with the dimensions $d = 10, 20,$ and $50$.
We observe that the number of parameters of HOrderDNN surges with increases in either $d$ or $p$, rendering it impractical for high-dimensional problems, particularly for $d = 20$ or $d = 50$. In contrast, K-HOrderDNNs achieve relative errors significantly lower than PINN across all dimensions, with their parameter count being either less than or slightly exceeding that of PINN. Additionally, as $p$ increases, K-HOrderDNN($p$) yields increasingly accurate numerical solutions. We observe that for $d$=50, the error of K-HOrderDNN($p=9$) is only reduced by 5.00E$-$1 compared to PINN, which is speculated to result from the insufficient sample size of 35,000. 

\begin{table}[ht]
\centering
\caption{Parameter settings of different methods on the problem (\ref{equation3}) with the dimension $d$ = 10, 20, and 50. Here, $L = hd+gd+1$, $W = max(hw, gw)$.}
\smallskip
\setlength{\tabcolsep}{3pt}
\begin{tabular}{ccccccccccc} 
\toprule
\multirow{2}[4]{*}{$d$} & \multicolumn{2}{c}{PINN} & \multicolumn{4}{c}{HOrderDNN} & \multicolumn{4}{c}{K-HOrderDNN} \\ \cmidrule{2-3} \cmidrule{4-5} \cmidrule{6-11}
& $W$  & $L$  & & & $W$  & $L$  & $hw$  & $hd$  & $gw$  & $gd$ \\ \midrule
10  & 210   & 4 & &  & 210   & 4    & 210  & 1   & 210 & 2 \\ 
20  & 205   & 4 & & & 205   & 4    & 205  & 1   & 205 & 2 \\  
50  & 202   & 4 &  & & 202   & 4    & 202  & 1   & 202 & 2 \\  
\bottomrule
\end{tabular}%
\label{tab:table5}%
\end{table}%

\begin{table}[ht]
	\centering
	\setlength{\tabcolsep}{3pt}
	\caption{The relative $L_2$ errors and the number of parameters of different methods on problem (\ref{equation3}) with the dimension $d$ = 10, 20, and 50.}
	\smallskip
	\begin{tabular}{cccccccc} 
		\toprule
		\multirow{2}{*}{Method} & \multirow{2}{*}{$p$} & \multicolumn{2}{c}{d=10} & \multicolumn{2}{c}{d=20} & \multicolumn{2}{c}{d=50}  \\
		\cmidrule{3-8}         
		&       & Params & RLE & Params &  RLE & Params &  RLE  \\ \midrule
		
		\multicolumn{2}{c}{PINN} \textbackslash{} & 1.3545E+05 & \textbf{5.76E-01} & 1.3120E+05 & \textbf{7.65E-01} & 1.3352E+05 & \textbf{8.51E-01}  \\
		\midrule
		\multirow{5}[3]{*}{HOrderDNN} & 1     & 3.4839E+05 & \textbf{9.92E-01} & 2.1509E+08 & \textbackslash{} & 2.2743E+17 & \textbackslash{}  \\
		& 3     & 2.2033E+08 & \textbackslash{} & 2.2540E+14 & \textbackslash{} & 2.5607E+32 & \textbackslash{}  \\
		& 5     & 1.2698E+10 & \textbackslash{} & 7.4951E+17 & \textbackslash{} & 1.6327E+41 & \textbackslash{}  \\
		& 7     & 2.2549E+11 & \textbackslash{} & 2.3635E+20 & \textbackslash{} & 2.8830E+47 & \textbackslash{}  \\
		& 9     & 2.1000E+12 & \textbackslash{} & 2.0500E+22 & \textbackslash{} & 2.0200E+52 & \textbackslash{}  \\
		\midrule
		\multirow{5}[3]{*}{K-HOrderDNN} & 1     & 9.3892E+04 & 4.31E-01 & 2.1980E+05 & 5.58E-01 & 1.0826E+06 & 5.31E-01  \\
		& 3     & 9.4312E+04 & 4.50E-02 & 2.2021E+05 & 1.15E-01 & 1.0830E+06 & 4.58E-01  \\
		& 5     & 9.4732E+04 & 4.64E-02 & 2.2062E+05 & 1.27E-01 & 1.0834E+06 & 3.51E-01  \\
		& 7     & 9.5152E+04 & 3.16E-02 & 2.2103E+05 & 3.15E-01 & 1.0838E+06 & 4.68E-01  \\
		& 9     & 9.5572E+04 & \textbf{2.78E-02} & 2.2144E+05 & \textbf{9.00E-02} & 1.0842E+06 & \textbf{3.46E-01}  \\
		\bottomrule
	\end{tabular}%
	\label{tab:table6}%
\end{table}%

\subsubsection{Experimental Study on Convergence Behavior}
We have conducted numerical experiments to verify the convergence rate using the ReLU and Tanh activation functions. For the ReLU activation function, we vary $n$ (where $n = gw$) and $ p $. The results are shown in panels (a) and (b) of Fig. \ref{fig:figRR1}. For panel (a), we fix $p = 40 $, $ hw = 45$, $ hd = 3 $, and $ gd = 2 $, while $gw$ takes values of (5,15,30,45,60). For panel (b), we set $hw=45$, $hd=3$, $gw=200$, $gd=2$, and let $p$ vary as $1,5,10,20,30,40$. In panel (a), as $n$ increases, the relative $L_2$ error decreases at the expected theoretical rate of $O(n^{-1})$. In panel (b), as $p$ increases, the relative $L_2$ error decays faster than the theoretical rate of $O(p^{-\alpha})$.
\par
For the Tanh activation function, vary $N$ (with $n = gw = 6N$) and $p $. The results are presented in panels (c) and (d) of Fig. \ref{fig:figRR1}. In panel (c), we fix $p=30$, $hw=45$, $hd=3$, and $gd=2$, while $N$ increases from 6 to 14 in steps 2, corresponding to gw values from 36 to 84 in steps 12. In panel (d), we set $ hw = 45 $, $hd = 3 $, $ gw = 250$, and $ gd = 2$, and let $ p$ vary as 1,5,10,20,30,40. In panel (c), as $N$ increases, the relative $L_2$ error decreases at the theoretical rate of $O(N^{-1})$. In panel (d), as $p$ increases, the error decays faster than $O(p^{-\alpha})$, similar to the behavior observed for the ReLU case in panel (b). This observation might be attributed to the potentially higher smoothness of the K-inner function.
\begin{figure}[H]
	\centering
	\includegraphics[width=0.8\textwidth]{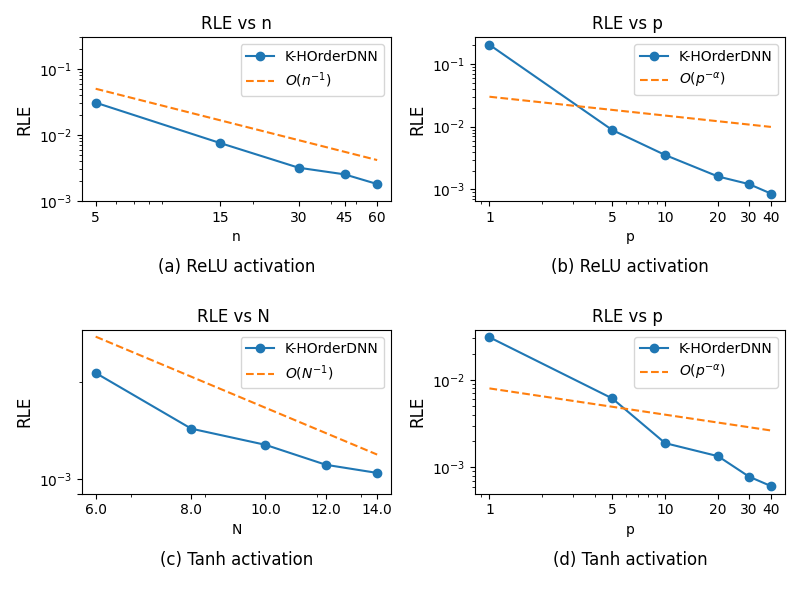}
	\caption{The relative $ L_2 $ error Trends for ReLU activation function (first row) and Tanh activation function(second row): (a) Loss vs $n$ for fixed $p = 40$; (b) Loss vs $p$ for fixed $n = 200$; (c) Loss vs $N$ for fixed $p = 30$ with $n = 6N$; (d) Loss vs $p$ for fixed $n = 250$.}
	\label{fig:figRR1}
\end{figure}
\par
\begin{figure}[ht]
	\centering
	\includegraphics[width=\textwidth]{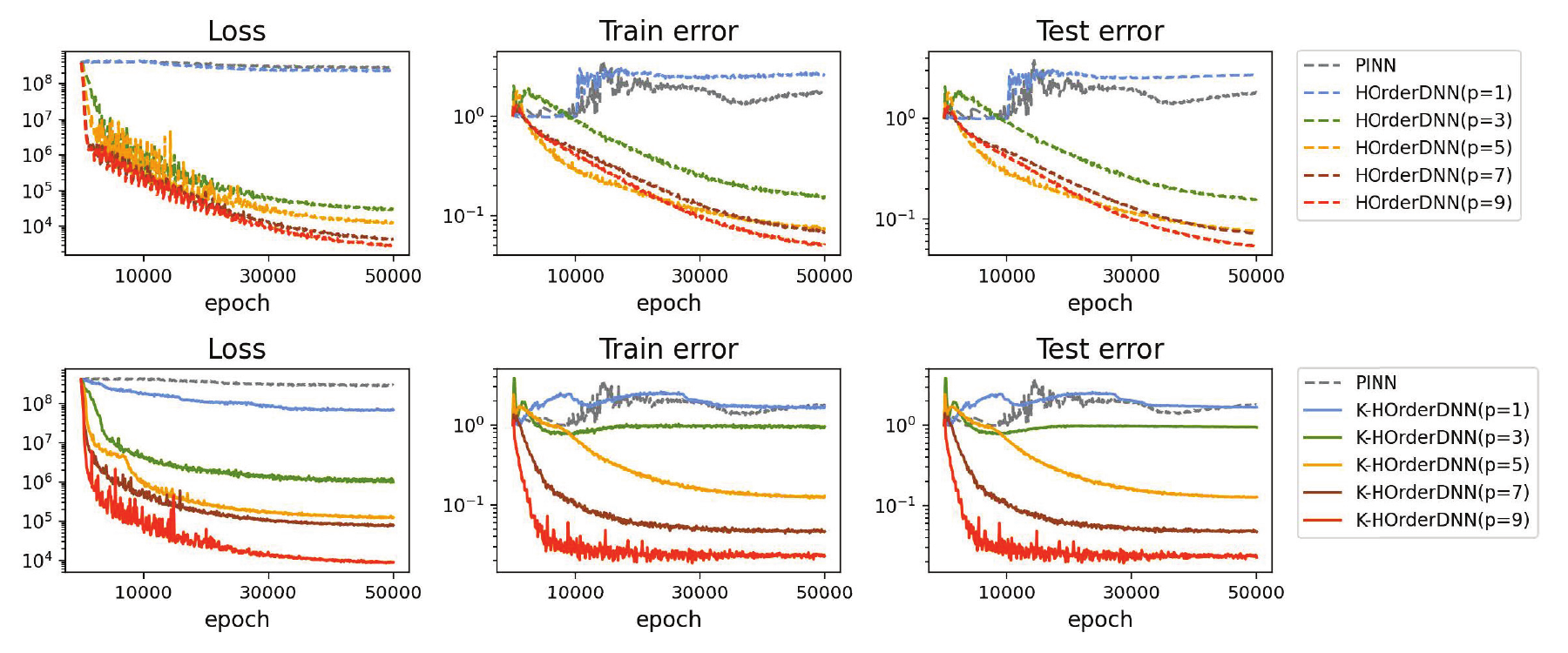}
	\caption{Convergence Processes of PINN, HOrderDNNs and K-HOrderDNNs on problem (\ref{equation4}) when $d$=2. Here, $hd$=3, $hw$=45, $gd$=2, $gw$=90 for K-HOrderDNNs, and $L=$6, $W=$90 for HOrderDNNs and PINN.}
	\label{fig:8}
\end{figure}
\subsection{Poisson equation}
\label{Poisson}
We consider the following Poisson equation with a Dirichlet condition in $\Omega = [0,1]^d \subset \mathbb{R}^d$
\begin{equation}
\left\{
\begin{aligned}
-\Delta u &= f, & \quad x &\in \Omega ,\\
u &= g, & \quad x &\in \partial \Omega,
\end{aligned}
\right.
\label{equation4}
\end{equation}
where the right-hand side $f$ and the boundary condition $g$ are derived from the exact solution. We use the loss function defined in Equation \eqref{eq:eq15} to train model and utilize the commonly adopted learning rate annealing methodology \cite{wang2021understanding} to adjust the loss weight parameter $\beta$. These techniques are also adopted when solving the Helmholtz equation. 

\subsubsection{Low-dimensional and high-frequency problem}
\label{case2d}
We first consider a two-dimensional problem with a high-frequency solution, represented by Eq. \eqref{equation2}. In this example, we choose the optimal hyperparameter settings in Subsection \ref{f1}. Specifically, for K-HOrderDNNs, we set $hw$ to 45, $hd$ to 3, $gd$ to 2, with $gw$ being set at twice the value of $hw$, and $p = 1, 3, 5, 7, 9$. For both HOrderDNNs and PINN, the network depth $L$ is set to 6, and the network width $W$ is aligned with $gw$. In each epoch, we resample $N_f = 5000$ , and $N_b = 1000$ training points. 
\par
\begin{figure}[ht]
	\centering
	\includegraphics[width=\textwidth]{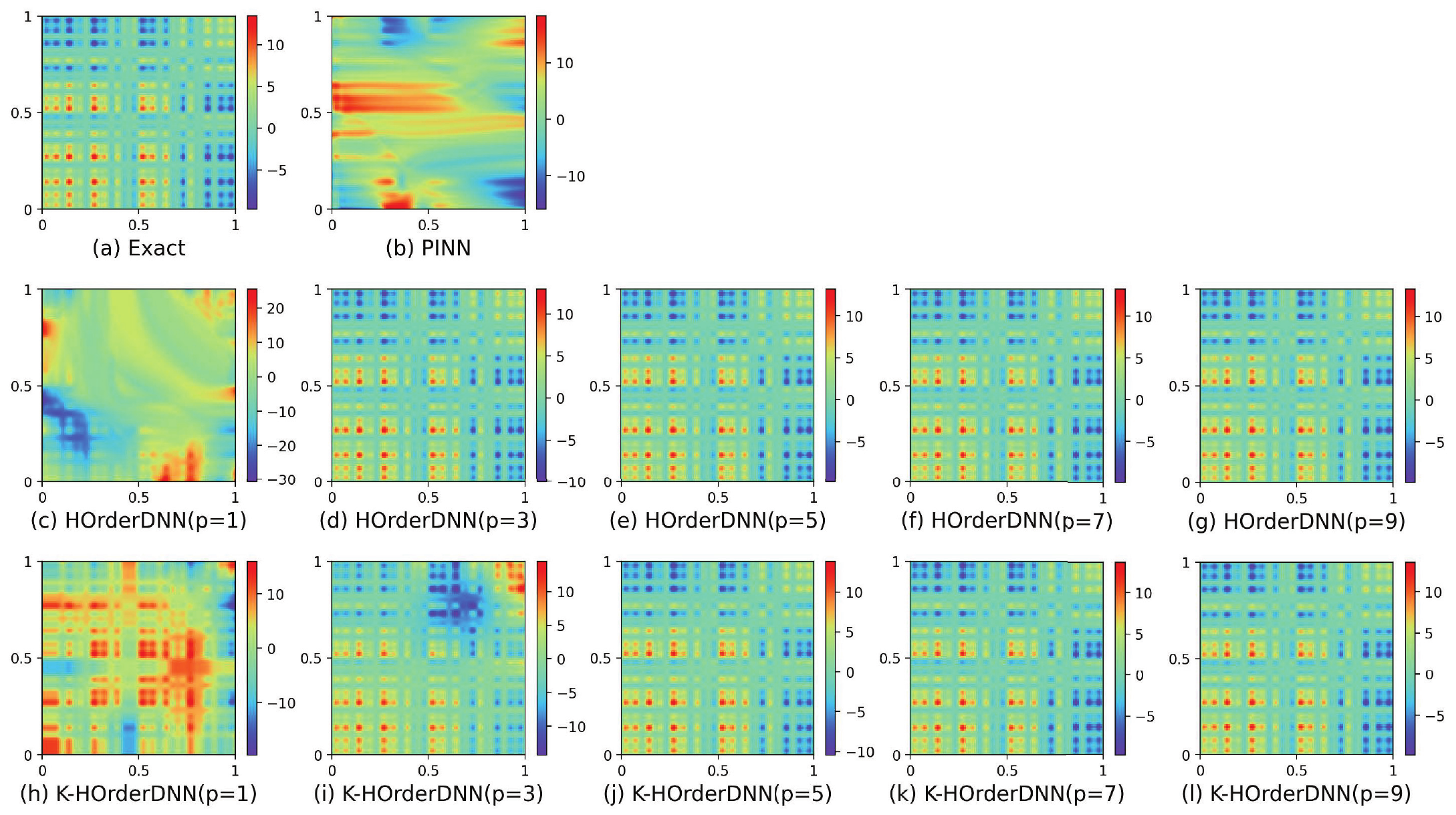}
	\caption{Exact ($a$) and numerical solutions of PINN ($b$), HOrderDNNs ($c-g$), and K-HOrderDNNs ($h-l$) on problem (\ref{equation4}) when $d$=2. Here, $hd$=3, $hw$=45, $gd$=2, $gw$=90 for K-HOrderDNNs, and $L$=6, $W$=220 for HOrderDNNs and PINN.}
	\label{fig:9}
\end{figure}
\par
\begin{table}[H]
	\centering
	\setlength{\tabcolsep}{3pt}
	\caption{The relative $L_2$ errors and the number of parameters obtained by PINN, HOrderDNNs, and K-HOrderDNNs on problem (\ref{equation4}) when $d$ = 2.}
	\smallskip
	\begin{tabular}{ccccccc}
		\toprule
		\multirow{2}[4]{*}{\diagbox{$p$}{Method}}  & \multicolumn{2}{c}{PINN} & \multicolumn{2}{c}{HOrderDNN} & \multicolumn{2}{c}{K-HOrderDNN} \\
		\cmidrule{2-7}          & Params & RLE   & Params & RLE   & Params & RLE \\
		\midrule
		\textbackslash{} & 4.1311E+04 & \textbf{9.78E-01} & \textbackslash{} & \textbackslash{} & \textbackslash{} & \textbackslash{} \\
		1     & \textbackslash{} & \textbackslash{} & 4.1491E+04 & 9.88E-01 & 1.3776E+04 & 1.01E+00 \\
		3     & \textbackslash{} & \textbackslash{} & 4.2571E+04 & 1.54E-01 & 1.3866E+04 & 7.74E-01 \\
		5     & \textbackslash{} & \textbackslash{} & 4.4371E+04 & 7.56E-02 & 1.3956E+04 & 1.26E-01 \\
		7     & \textbackslash{} & \textbackslash{} & 4.6891E+04 & 7.08E-02 & 1.4046E+04 & 4.63E-02 \\
		9     & \textbackslash{} & \textbackslash{} & 5.0131E+04 & \textbf{5.37E-02} & 1.4136E+04 & \textbf{1.91E-02} \\
		\bottomrule
	\end{tabular}%
	\label{tab:table7}%
\end{table}%
Table \ref{tab:table7} presents a comparative analysis of different models in terms of the number of parameters and the minimum $L_2$ relative errors. As we can see, K-HOrderDNN($p$) demonstrates superior performance over both the PINN and HOrderDNN($p$) at higher orders. Notably, K-HOrderDNN($p=9$) achieves the lowest relative $L_2$ error (1.91E-02) with 1.4136E+04 parameters, while HOrderDNN($p=9$) exhibits a higher relative $L_2$ error (5.37E-02) despite a significantly larger parameter count (5.0131E+04). The convergence behaviors of the PINN, HOrderDNNs, and K-HOrderDNNs are illustrated in Fig. \ref{fig:8}. K-HOrderDNN with larger $p$ converges faster than HOrderDNN, achieving lower errors, while the errors of PINN remain almost unchanged. The exact and numerical solutions for different models are depicted in Fig. \ref{fig:9}. We observe that PINN, HOrderDNN($p=1$), and K-HOrderDNN($p=1$) cannot fit local oscillations in the exact solution, while both HOrderDNNs and K-HOrderDNNs with the order $p>3$ successfully capture these oscillations, and the effects improve with increasing $p$. This illustrates that both HOrderDNNs and K-HOrderDNNs are capable of capturing oscillations at various scales when solving high-frequency PDEs.

\par
We also consider a Poisson problem \eqref{equation4} with a corner singular solution in an L-shaped domain $\Omega = (-1, 1)^2 \setminus [0, 1] \times [-1, 0]$. The exact solution is chosen as $ r^{\frac{2}{3}} \sin\left(\frac{2}{3} \theta\right) + x_2 \cos(4\pi(x_1 + 2x_2))$, where $(r, \theta)$ is the polar coordinate at the origin $O=(0,0)$. Notably, $u(x_1, x_2)$ has a corner singularity at the origin, characterized by $u \in H^{\frac{5}{3}-\epsilon} \not\subset H^2(\Omega)$ for $\epsilon > 0$\cite{zeng2022adaptive}. The right-hand side $f$ and the Dirichlet boundary function $g$ can can both be derived from the exact solution. We evaluate the performance of our K-HOrderDNNs and compare it with PINN and HOrderDNNs. For each method, we resample $N_f = 6000$ and $N_b = 400$ training points in each epoch, where $N_b$ is propositional to their lengths. The K-HOrderDNNs are configured with hyperparameters $hw = 90$, $hd = 3$, $gd = 3$, $gw = 90$, and $p=1, 3, 5, 7, 9$. Both HOrderDNNs and PINN share the same settings for $L = 7$ and $W = 90$, but HOrderDNNs also include $p=1, 3, 5, 7, 9$.
\par
\begin{table}[htbp]
	\centering
	\setlength{\tabcolsep}{3pt}
	\caption{The relative $L_2$ errors and the number of parameters obtained by PINN, HOrderDNNs, and K-HOrderDNNs on problem \eqref{equation4} in an L-shaped domain.}
	\smallskip
	\begin{tabular}{ccccccc}
		\toprule
		\multirow{2}[4]{*}{\diagbox{$p$}{Method}}  & \multicolumn{2}{c}{PINN} & \multicolumn{2}{c}{HOrderDNN} & \multicolumn{2}{c}{K-HOrderDNN} \\
		\cmidrule{2-7}          & Params & RLE   & Params & RLE   & Params & RLE \\
		\midrule
		\textbackslash{} & 4.9501E+04 & \textbf{4.36E-02} & \textbackslash{} & \textbackslash{} & \textbackslash{} & \textbackslash{} \\
		1     & \textbackslash{} & \textbackslash{} & 4.9681E+04 & 3.05E-02 & 3.4566E+04 & 1.15E-01 \\ 
		3     & \textbackslash{} & \textbackslash{} & 5.0761E+04 & 4.03E-03 & 3.4746E+04 & 4.23E-03 \\ 
		5     & \textbackslash{} & \textbackslash{} & 5.2561E+04 & 3.17E-03 & 3.4926E+04 & 4.17E-03 \\
		7     & \textbackslash{} & \textbackslash{} & 5.5081E+04 & 3.48E-03 & 3.5106E+04 & 4.37E-03 \\
		9     & \textbackslash{} & \textbackslash{} & 5.8321E+04 & \textbf{2.56E-03} & 3.5286E+04 & \textbf{2.70E-03} \\
		\bottomrule
	\end{tabular}%
	\label{tab:RR1}%
\end{table}%
\par
In Fig. \ref{fig:figR3}, we present the exact and the absolute pointwise errors for different method, with their relative $L_2$ errors summarized in Table \ref{tab:RR1}. As we can see, K-HOrderDNNs($p>1$) achieve comparable performance to HOrderDNNs($p>1$) using fewer parameters. Both methods significantly outperform PINN, achieving an order-of-magnitude reduction in error. In conclusion, K-HOrderDNNs, which utilize fewer parameters, demonstrate comparable accuracy to HOrderDNNs in approximating a less smooth solution within an L-shaped domain.
\begin{figure}[H]
	\centering
	\includegraphics[width=\textwidth]{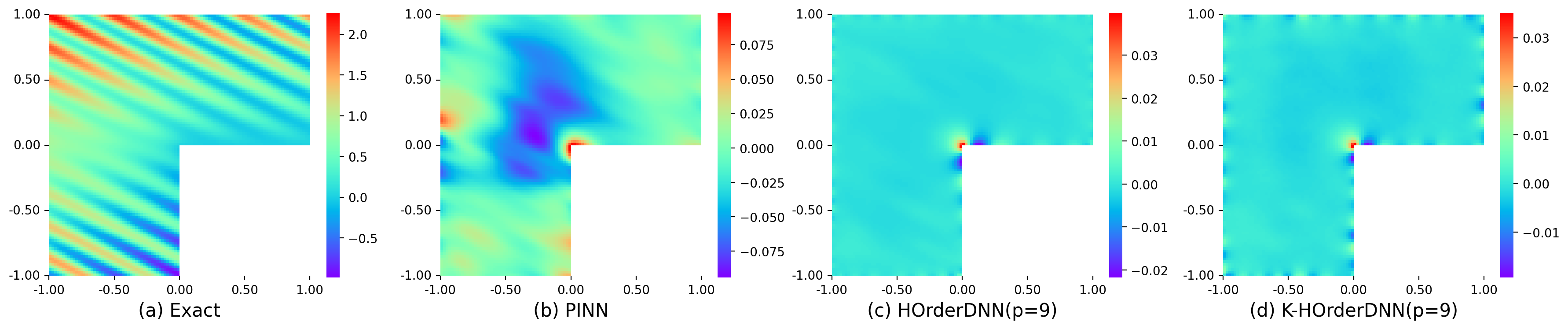}
	\caption{Exact solution and absolute pointwise errors of PINN, HOrderDNN(9), and K-HOrderDNN(9) in an L-shaped domain. Here, hd=3, hw=90, gd=3, gw=90 for K-HOrderDNN(9), and L =7, W =90 for HOrderDNN(9) and PINN.}
	\label{fig:figR3}
\end{figure}
\par

\subsubsection{High-dimensional and low-frequency problem}
\label{PE:hdlf}
Next, we turn our attention to a high-dimensional problem with a low-frequency exact solution, which is defined as follows:
\begin{equation}
u(x_1, x_2, \ldots, x_{10}) = \prod_{i=1}^{10} \sin(\pi x_i), \quad  x_i \in [0, 1].
\label{equation5}
\end{equation}
\par
We use the network architectures of PINN, HOrderDNNs, and K-HOrderDNNs as described in Table \ref{tab:table5} for the case when d=10 to obtain the numerical solutions. In each epoch, we resample $N_f = 8000$ and $N_b = 2000$. In subsequent experiments, this configuration is used for all high-dimensional problems.
\par
\begin{table}[htbp]
	\centering
	\setlength{\tabcolsep}{3pt}
	\caption{The relative $L_2$ errors and the number of parameters obtained by PINN, HOrderDNNs, and K-HOrderDNNs for problem (\ref{equation4}) when $d$=10 with a solution represented by Eq. (\ref{equation5}).}
	\smallskip
	\begin{tabular}{ccccccc}
		\toprule
		\multirow{2}[4]{*}{\diagbox{$p$}{Method}}  & \multicolumn{2}{c}{PINN} & \multicolumn{2}{c}{HOrderDNN} & \multicolumn{2}{c}{K-HOrderDNN} \\
		\cmidrule{2-7}          & Params & RLE   & Params & RLE   & Params & RLE \\
		\midrule
		\textbackslash{} & 1.3545E+05 & \textbf{2.50E-01} & \textbackslash{} & \textbackslash{} & \textbackslash{} & \textbackslash{} \\
		1     & \textbackslash{} & \textbackslash{} & 3.4839E+05 & \textbf{1.37E+00} & 9.3892E+04 & 1.92E-01 \\
		3     & \textbackslash{} & \textbackslash{} & 2.2033E+08 & \textbackslash{} & 9.4312E+04 & 4.03E-02 \\
		5     & \textbackslash{} & \textbackslash{} & 1.2698E+10 & \textbackslash{} & 9.4732E+04 & 5.65E-02 \\
		7     & \textbackslash{} & \textbackslash{} & 2.2549E+11 & \textbackslash{} & 9.5152E+04 & \textbf{3.37E-02} \\
		9     & \textbackslash{} & \textbackslash{} & 2.1000E+12 & \textbackslash{} &9.5572E+04 & 4.65E-02 \\
		\bottomrule
	\end{tabular}%
	\label{tab:table8}%
\end{table}%
\par
The relative $L_2$ errors and number of parameters are presented in Table \ref{tab:table8}, where we can see that K-HOrderDNNs achieve the highest accuracy among all the methods. Particularly noteworthy is that the number of parameters of HOrderDNN($p$) surges significantly as the order $p$ rises, which makes the computation of HOrderDNNs($p > 1$) impractical. In contrast, K-HOrderDNN($p$) demonstrates a more gradual increase in parameters, ensuring computational efficiency. We present the absolute pointwise errors obtained by PINN, HOrderDNN($p=1$), and K-HOrderDNNs in Fig. \ref{fig:10}. Due to the difficulty of visualizing high-dimensional spaces,  we only present points on a slice defined by the first two dimensions, $x_1$ and $x_2$. As expected, PINN and HOrderDNN($p=1$) fail to capture the exact solution, while K-HOrderDNNs($p>1$) succeed, with its ability enhancing as $p$ increases. This observation demonstrates that K-HOrderDNNs still perform well for low-frequency, high-dimensional problems.
\par
\begin{figure}[ht]
\centering
\includegraphics[width=\textwidth]{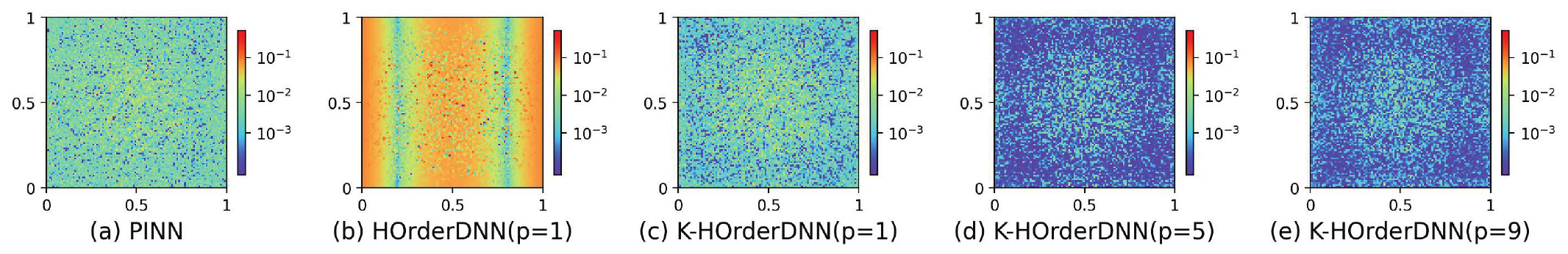}
\caption{Comparative analysis of absolute pointwise errors for PINN, HOrderDNN($p=1$), and K-HOrderDNNs on a 2D Slice with a exact solution represented by Eq. (\ref{equation5}).}
\label{fig:10}
\end{figure}

\subsubsection{High-dimensional and high-frequency problem}
Finally, we consider three ten-dimensional, high-frequency test problems:
\par
\textbf{Test problem 1:} The exact solution is an additive composition of high-frequency components, expressed as: 
\begin{equation}
u(x_1, x_2, \cdots, x_{d}) = \sum_{i=1}^{10} \sum_{j=0}^{5} \sin(2^j \pi x_i), \quad \mathrm{x}_i \in [0,1] .
\label{equation6}
\end{equation}
\par
\textbf{Test problem 2:} We consider a complex exact solution with enhanced frequency interactions, formulated as:  
\begin{equation}
\begin{aligned}
u(x_1, x_2, \ldots, x_d) = \sum_{i=1}^{8} \sin(10 \pi x_i) \cdot \sin(10 \pi x_{i+1}) \cdot \sin(10 \pi x_{i+2}), \quad x_i \in [0, 1].
\end{aligned}
\label{equation7}
\end{equation}
\par
\textbf{Test problem 3:} We further test a more complex exact solution that integrates frequency variation and intricate interactions, which is defined as follows:
\begin{equation}
\begin{aligned}
u(x_1, x_2, \ldots, x_d) = \sum_{i=1}^{8} g(x_i) \cdot g(x_{i+1}) \cdot g(x_{i+2}), \quad x_i \in [0, 1],
\end{aligned}
\label{equation8}
\end{equation}
where $g(x_i)$ is defined in Equation \eqref{equation2}.

Tables \ref{tab:table9} presents the relative $L_2$ errors and the number of parameters for the above three test problems. In agreement with observations in the Subsection \ref{PE:hdlf}, in all cases, K-HOrderDNNs with larger $p$ exhibit superior performance over both HorderDNN($p=1$) and PINN, as evidenced by its lower relative $L_2$ errors with fewer parameters. Unlike HorderDNNs($p>1$), which struggle with the CoD issue, K-HOrderDNNs maintain computational feasibility. The absolute pointwise errors for different models of the above three cases are depicted in Fig. \ref{fig:11}. These figures show that PINN, HorderDNN($p=1$), and K-HOrderDNN($p$) with a smaller value of $p$ are unable to fit oscillations caused by high frequencies, whereas K-HOrderDNN($p$) with a larger value of $p$ successfully captures these oscillations. Furthermore, the approximation capabilities of K-HOrderDNN($p$) improve as $p$ increases. These observations demonstrate K-HOrderDNN's superior capability to capture the high-frequency information of solutions when solving high-dimensional and high-frequency PDEs. 
\par
\begin{table}[htbp]
	\centering
	\setlength{\tabcolsep}{3pt}
	\caption{The relative $L_2$ errors and the number of parameters obtained by PINN, HOrderDNNs, and K-HOrderDNNs for Test Problems 1-3 when $d$=10.}
	\smallskip
	\begin{tabular}{cccccc}
		\toprule
		\multirow{2}[4]{*}{Method } & \multirow{2}[4]{*}{ $p$ } & \multirow{2}[4]{*}{ Params } & \multicolumn{3}{c}{REL} \\
		\cmidrule{4-6}          &       &       & Test problem 1 & Test problem 2 & Test problem 3 \\
		\midrule
		PINN   & \textbackslash{} & 1.3545E+05 & \textbf{5.21E-01} & \textbf{1.00E+00} & \textbf{9.91E-01} \\
		\midrule
		\multirow{5}[2]{*}{HOrderDNN} & 1     & 3.4839E+05 & \textbf{5.41E-01} & \textbf{1.00E+00} & \textbf{9.93E-01} \\
		& 3     & 2.2033E+08 & \textbackslash{} & \textbackslash{} & \textbackslash{} \\
		& 5     & 1.2698E+10 & \textbackslash{} & \textbackslash{} & \textbackslash{} \\
		& 7     & 2.2549E+11 & \textbackslash{} & \textbackslash{} & \textbackslash{} \\
		& 9     & 2.1000E+12 & \textbackslash{} & \textbackslash{} & \textbackslash{} \\
		\midrule
		\multirow{5}[2]{*}{K-HOrderDNN} & 1     & 9.3892E+04 & 5.59E-01 & 1.00E+00 & 9.92E-01 \\
		& 3     & 9.4312E+04 & 3.60E-01 & 1.00E+00 & 8.63E-02 \\
		& 5     & 9.4732E+04 & 1.52E-03 & 1.00E+00 & \textbf{1.27E-02} \\
		& 7     & 9.5152E+04 & 1.65E-03 & 4.55E-01 & 2.30E-02 \\
		& 9     & 9.5572E+04 & \textbf{1.16E-03} & \textbf{3.09E-02} & 1.30E-02 \\
		\bottomrule
	\end{tabular}%
	\label{tab:table9}%
\end{table}%

\subsection{Helmholtz equation}
\label{HE}
We consider the following Helmholtz equation with a Dirichlet condition in domain $\Omega = [0,1]^d \subset \mathbb{R}^d$
\begin{equation}
\left\{
\begin{aligned}
\Delta u + k^2 u &= f, & \quad x &\in \Omega ,\\
u &= g, & \quad x &\in \partial \Omega,
\end{aligned}
\right.
\label{equation9}
\end{equation}
where k=5, and right-hand side $f$ and the boundary condition $g$ are computed from the exact solution.

\subsubsection{Low-dimensional and high-frequency problem}
We first consider a two-dimensional problem with the following high frequency exact solution:
\begin{equation}
u(x_1, x_2) = \sin(25\pi x_1)\sin(25 \pi x_2), \quad \mathrm{x}_i \in [0,1].
\label{equation10}
\end{equation}
\par
\begin{figure}[H]
	\centering
	
	\subfigure[Test problem 1]{
		\includegraphics[width=\linewidth]{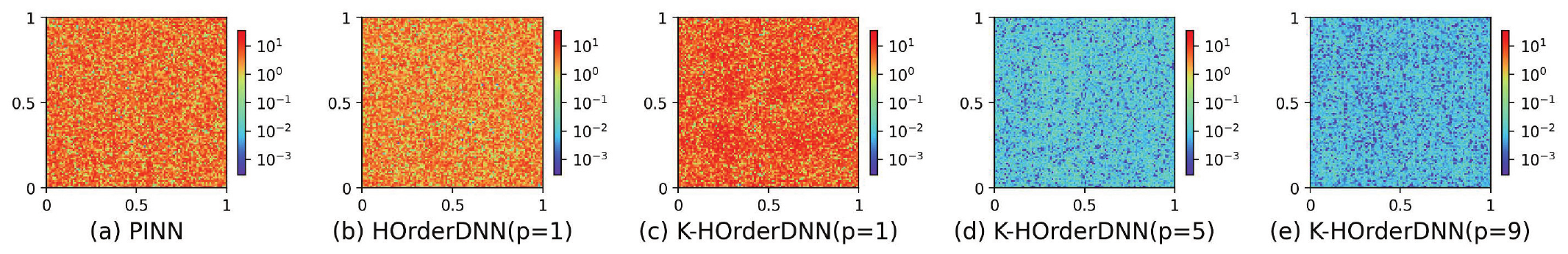}}
	
	\subfigure[Test problem 2]{
		\includegraphics[width=\linewidth]{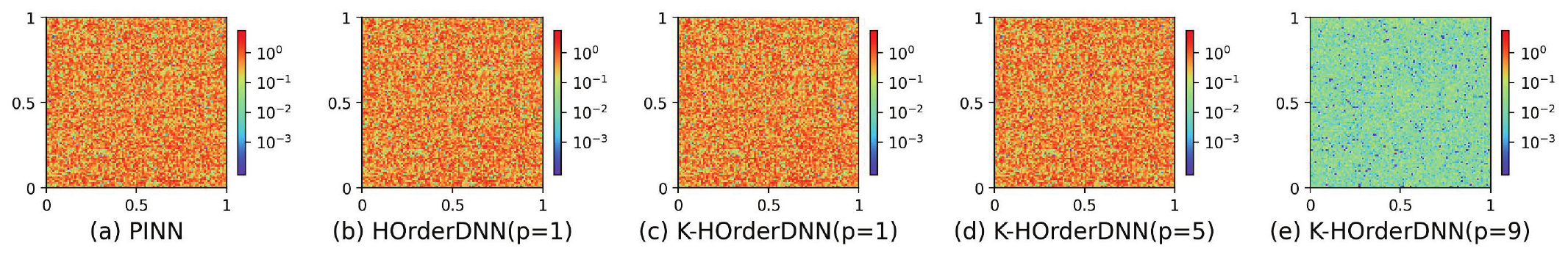}}
	
	\subfigure[Test problem 3]{
		\includegraphics[width=\linewidth]{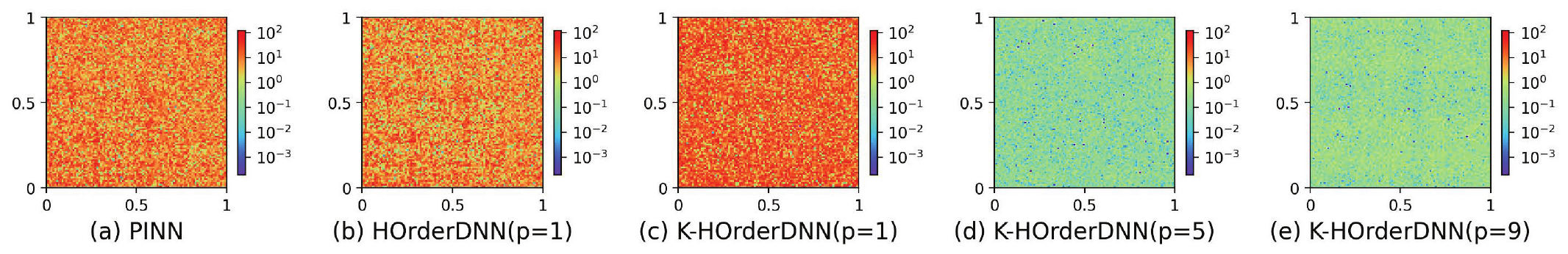}}
	
	\caption{The absolute pointwise errors obtained by PINN, HOrderDNN($p=1$), and K-HOrderDNNs on a 2D slice for Test Problem 1(\uppercase\expandafter{\romannumeral1}), Test Problem 2(\uppercase\expandafter{\romannumeral2}), and Test Problem 3(\uppercase\expandafter{\romannumeral3}).}
	
	\label{fig:11}
	
	%
	%
\end{figure}
\par
\begin{table}[H]
	\centering
	\setlength{\tabcolsep}{3pt}
	\caption{The relative $L_2$ errors and the number of parameters obtained by PINN, HOrderDNNs, and K-HOrderDNNs for problem \eqref{equation9} when $d$ = 2.}
	\smallskip
	\begin{tabular}{ccccccc}
		\toprule
		\multirow{2}[4]{*}{\diagbox{$p$}{Method}}  & \multicolumn{2}{c}{PINN} & \multicolumn{2}{c}{HOrderDNN} & \multicolumn{2}{c}{K-HOrderDNN} \\
		\cmidrule{2-7}          & Params & RLE   & Params & RLE   & Params & RLE \\
		\midrule
		\textbackslash{} & 4.1311E+04 & \textbf{1.00E+00} & \textbackslash{} & \textbackslash{} & \textbackslash{} & \textbackslash{} \\
		1     & \textbackslash{} & \textbackslash{} & 4.1491E+04 & 9.97E-01 & 1.3776E+04 & 1.00E+00 \\
		3     & \textbackslash{} & \textbackslash{} & 4.2571E+04 & 9.76E-01 & 1.3866E+04 & 9.02E-02 \\
		5     & \textbackslash{} & \textbackslash{} & 4.4371E+04 & 7.44E-01 & 1.3956E+04 & 7.87E-02 \\
		7     & \textbackslash{} & \textbackslash{} & 4.6891E+04 & 2.47E-01 & 1.4046E+04 & \textbf{7.07E-03} \\
		9     & \textbackslash{} & \textbackslash{} & 5.0131E+04 & \textbf{1.23E-01} & 1.4136E+04 & 1.19E-01 \\
		\bottomrule
	\end{tabular}%
	\label{tab:table12}%
\end{table}%
\par
In this example, we employ the hyperparameter settings outlined in Subsection \ref{case2d}. Table \ref{tab:table12} presents a comparison of the relative $L_2$ errors and the number of parameters for different models. Similar to the results in Subsection \ref{case2d}, K-HOrderDNNs show superior performance compared to HOrderDNNs and PINN. In detail, PINN has a relatively high error of 1.00E+00, indicating that it struggles with high-frequency problems, while K-HOrderDNN($p=7$) achieves the 7.07E-03 errors with the 1.4046E+04 parameters, three orders of magnitude smaller than PINN and two orders smaller than HOrderDNNs. Note that K-HOrderDNN($p=9$) performs worse than K-HOrderDNN($p=7$), which may be attributed to a more complex optimization landscape, leading to the optimization algorithm becoming trapped in suboptimal local minima. We adjust the learning rate from 0.004 to 0.006 to investigate this hypothesis, which results in a significant reduction in the relative $L_2$ error from 1.19E-01 to 1.22E-02.
\par
The convergence behaviors of different models are illustrated in Fig. \ref{fig:14}. It is clearly observed that K-HOrderDNNs($p >1$) still exhibit significantly faster convergence and achieve smaller errors compared to both HOrderDNNs($p >1$) and PINN. Furthermore, the exact and numerical solutions for different models are shown in Fig. \ref{fig:15}, where HOrderDNN($p=9$) and K-HOrderDNNs($p>1$) can capture local oscillations in target functions effectively while PINN, HOrderDNNs($p<9$) and K-HOrderDNN($p=1$) can not.
\par
\begin{figure}[ht]
	\centering
	\includegraphics[width=\textwidth]{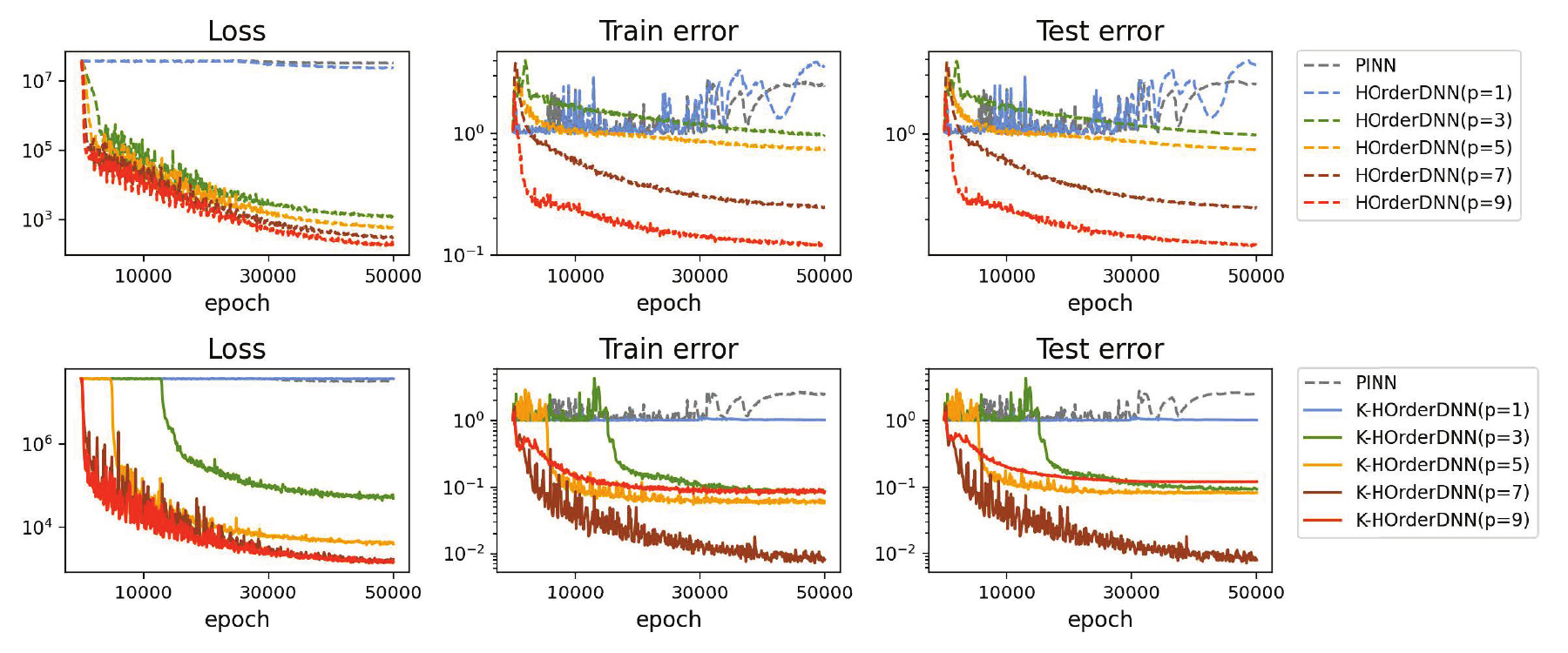}
	\caption{Convergence processes of PINN, HOrderDNNs and K-HOrderDNNs on problem (\ref{equation9}) when $d$=2. Here, $hd$=3, $hw$=45, $gd$=2, $gw$=90 for K-HOrderDNNs, and $L$=6, $W$=90 for HOrderDNNs and PINN.}
	\label{fig:14}
\end{figure} 
\par
\begin{figure}[ht]
	\centering
	\includegraphics[width=\textwidth]{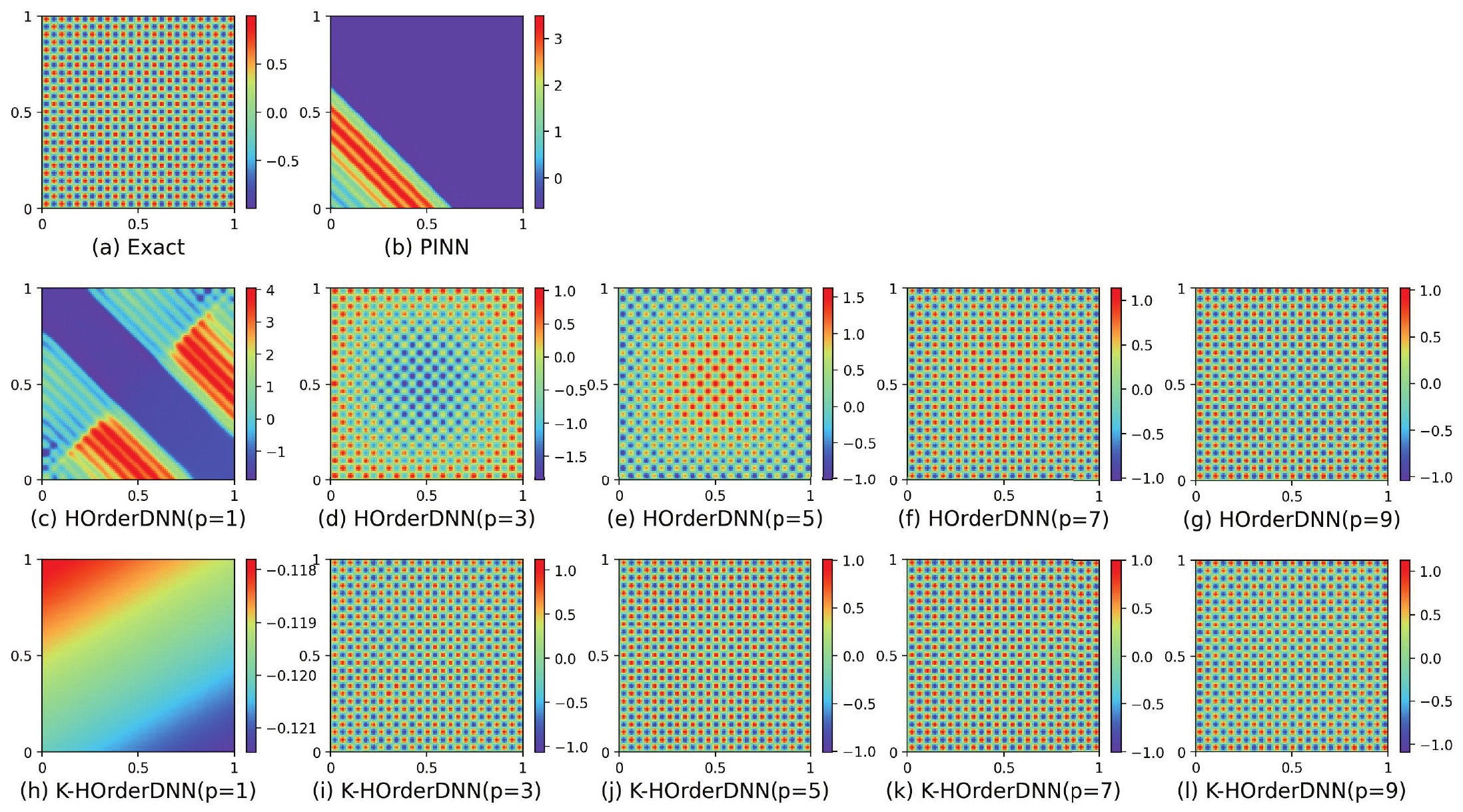}
	\caption{Exact solution ($a$) and numerical solutions given by PINN ($b$), HOrderDNNs (the second row), K-HOrderDNNs (the third row) on problem \eqref{equation9} when d=2. Here, $hd$=3, $hw$=45, $gd$=2, $gw$=90 for K-HOrderDNNs, and $L$=6, $W=$220 for HOrderDNNs and PINN.}
	\label{fig:15}
\end{figure}

\subsubsection{High-dimensinal and low-frequency problem}
Next, we turn to a high-dimensional problem with a low-frequency exact solution defined in Eq. \eqref{equation5}. The results for this example are summarized in Table \ref{tab:table13}. As expected, K-HOrderDNNs($p>1$) demonstrate superior performance compared to both PINN and HOrderDNN($p = 1$), especially at order 3, where it achieves the lowest error of 5.73E-02. Similar to the observations in Subsection \ref{PE:hdlf}, HOrderDNNs become intractable due to parameter explosion, while the number of parameters in K-HOrderDNNs increases moderately. Fig. \ref{fig:16} depicts the absolute pointwise errors for PINN, HOrderDNN($p=1$), and K-HOrderDNNs, where K-HOrderDNN($p$) with a larger $p$ can obtain more accurate results. This further validates the capability of K-HOrderDNN($p$) to efficiently solve high-dimensional PDEs. 
\begin{table}[ht]
	\centering
	\setlength{\tabcolsep}{3pt}
	\caption{The relative $L_2$ errors and the number of parameters obtained by PINN, HOrderDNNs, and K-HOrderDNNs for problem \eqref{equation9} when $d$=10 with a solution represented by Eq. \eqref{equation5}.}
	\smallskip
	\begin{tabular}{ccccccc}
		\toprule
		\multirow{2}[4]{*}{\diagbox{$p$}{Method}}  & \multicolumn{2}{c}{PINN} & \multicolumn{2}{c}{HOrderDNN} & \multicolumn{2}{c}{K-HOrderDNN} \\
		\cmidrule{2-7}          & Params & RLE   & Params & RLE   & Params & RLE \\
		\midrule
		\textbackslash{} & 1.3545E+05 & \textbf{2.16E-01} & \textbackslash{} & \textbackslash{} & \textbackslash{} & \textbackslash{} \\
		1     & \textbackslash{} & \textbackslash{} & 3.4839E+05 & \textbf{1.56E+00} & 9.3892E+04 & 2.35E-01 \\
		3     & \textbackslash{} & \textbackslash{} & 2.2033E+08 & \textbackslash{} & 9.4312E+04 & \textbf{5.73E-02} \\
		5     & \textbackslash{} & \textbackslash{} & 1.2698E+10 & \textbackslash{} & 9.4732E+04 & 9.90E-02 \\
		7     & \textbackslash{} & \textbackslash{} & 2.2549E+11 & \textbackslash{} & 9.5152E+04 & 7.46E-02  \\
		9     & \textbackslash{} & \textbackslash{} & 2.1000E+12 & \textbackslash{} &9.5572E+04 & 7.71E-02 \\
		\bottomrule
	\end{tabular}%
	\label{tab:table13}%
\end{table}%
\begin{figure}[ht]
\centering
\includegraphics[width=\textwidth]{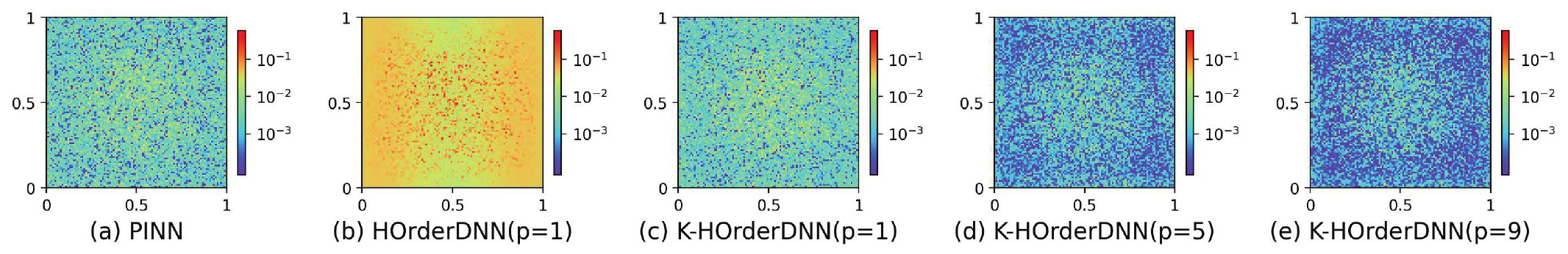}
\caption{Comparative analysis of absolute pointwise errors for PINN, HOrderDNN($p=1$), and K-HOrderDNNs on a 2D Slice with a exact solution represented by Eq. (\ref{equation5}).}
\label{fig:16}
\end{figure}

\subsubsection{High-dimensinal and high-frequency problem}
Lastly, we examine a high-dimensional problem with a high-frequency solution outlined in Eq. \eqref{equation8}. Table \ref{tab:table14} displays the relative $L_2$ errors and the number of parameters.
It is evident that K-HOrderDNNs($p>1$) outperform both PINN and HOrderDNN($p=1$), with the relative $L_2$ errors of K-HOrderDNN($p$) decreasing significantly as $p$ increases, while HOrderDNNs($p>1$) continue to exhibit limitations. Furthermore, the absolute pointwise errors of PINN, HOrderDNN($p=1$), and K-HOrderDNN($p$) are illustrated in Fig. \ref{fig:17}. As observed previously, PINN, HOrderDNN($p=1$), and K-HOrderDNN($p=1$) fail to approximate high-frequency solutions effectively, while K-HOrderDNNs($p>1$) capture them effectively. This demonstrates the superiority of the proposed K-HOrderDNNs in solving high-frequency PDEs in high dimensions.
\par
\begin{table}[htbp]
	\centering
	\setlength{\tabcolsep}{3pt}
	\caption{The relative $L_2$ errors and the number of parameters obtained by PINN, HOrderDNNs, and K-HOrderDNNs for problem \eqref{equation9} when $d$=10 with a solution represented by Eq. \eqref{equation8}.}
	\smallskip
	\begin{tabular}{ccccccc}
		\toprule
		\multirow{2}[4]{*}{\diagbox{$p$}{Method}}  & \multicolumn{2}{c}{PINN} & \multicolumn{2}{c}{HOrderDNN} & \multicolumn{2}{c}{K-HOrderDNN} \\
		\cmidrule{2-7}          & Params & RLE   & Params & RLE   & Params & RLE \\
		\midrule
		\textbackslash{} & 1.3545E+05 & \textbf{9.93E-01} & \textbackslash{} & \textbackslash{} & \textbackslash{} & \textbackslash{} \\
		1     & \textbackslash{} & \textbackslash{} & 3.4839E+05 & \textbf{9.93E-01} & 9.3892E+04 & 9.93E-01 \\
		3     & \textbackslash{} & \textbackslash{} & 2.2033E+08 & \textbackslash{} & 9.4312E+04 & 2.12E-02 \\
		5     & \textbackslash{} & \textbackslash{} & 1.2698E+10 & \textbackslash{} & 9.4732E+04 & 6.23E-03 \\
		7     & \textbackslash{} & \textbackslash{} & 2.2549E+11 & \textbackslash{} & 9.5152E+04 & \textbf{4.40E-03}  \\
		9     & \textbackslash{} & \textbackslash{} & 2.1000E+12 & \textbackslash{} &9.5572E+04 & 6.05E-03 \\
		\bottomrule
	\end{tabular}%
	\label{tab:table14}%
\end{table}%
\par
\begin{figure}[ht]
\centering
\includegraphics[width=\textwidth]{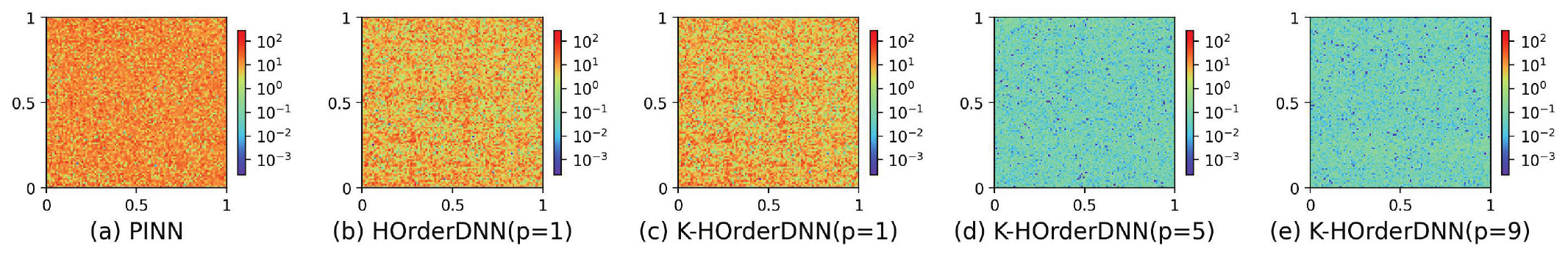}
\caption{Comparative analysis of absolute pointwise errors for PINN, HOrderDNN($p=1$), and K-HOrderDNN($p$) on a 2D Slice with a exact solution represented by Eq. (\ref{equation8}).}
\label{fig:17}
\end{figure}
\section{Discussion}
We compare our approach with several state-of-the-art methods, including Tensor Neural Networks (TNN) \cite{wang2024solving} and Kolmogorov–Arnold Networks (KAN) \cite{liu2024kan}, which have shown promise in high-dimensional function approximation. Ziming Liu et al. \cite{liu2024kan}, also inspired by the Kolmogorov-Arnold Representation Theorem (KAT), which was later refined into the KST, recently proposed KAN for function approximation. Both KAN and K-HOrderDNN are inspired by the KAT or KST, but they differ in how they model univariate K-inner and K-outer functions. In KAN, the univariate K-inner and K-outer functions are parametrized as B-spline curves, with learnable coefficients for the local B-spline basis functions. However, the prototype of KAN is relatively simple, limiting its ability to approximate complex functions, especially when smooth splines are involved. To address this, KAN is extended to deeper and wider architectures to enhance its approximation capabilities. In contrast, K-HOrderDNN approximates the univariate K-inner functions using HOrderDNNs and models the univariate K-outer function with a fully connected subnetwork. Experimentally, K-HOrderDNN achieved comparable accuracy to KAN in a 10-dimensional high-frequency function fitting task, but with about 12 times faster training. When solving the 2D high-frequency Poisson equation, K-HOrderDNN outperformed KAN in accuracy, while requiring about one-fifth of the training time. A detailed comparison of the architecture and experimental results for KAN and K-HOrderDNN can be found in Appendix C. TNN adopts a fundamentally different approach, leveraging a low-rank tensor structure to approximate a high-dimensional function as a sum of rank-one components formed by tensor products of one-dimensional functions. This structure reduces high-dimensional integrals in the loss function to one-dimensional integrals, enabling efficient computation via classical quadrature schemes. However, TNN is limited to problems where the coefficients and source terms exhibit a tensor-product structure, where K-HOrderDNN not restricted to this structure. Detailed the architecture and experimental results of TNN and HOrderDNN are provided in Appendix D.

\section{Conclusion} 
Inspired by KST, we have introduced K-HOrderDNN to solve the CoD problem suffered by HOrderNN when solving high frequency PDEs in high dimensional space. K-HOrderDNN utilizes an HOrderDNN to efficiently approximate univariate inner functions and employs a fully connected subnetwork to simulate the univariate outer function, which are integrated via the KST framework to approximate a multivariate function. We have established approximation rates for K-HOrderDNN over a dense subset of continuous multivariate functions, and demonstrated its ability to bypass the CoD. Through extensive numerical experiments on high frequency function fitting problems, high frequency Poisson and Helmholtz equations, especially in high dimensional problems, we have demonstrated that K-HOrderDNN($p>1$) not only further enhances the capabilities of HOrderDNN($p$) in tackling high frequency challenges, but also solves the parameter explosion problem encountered by HOrderDNN($p>1$) in high dimensional problems, while still surpassing the performance of PINN. For future work, we plan to extend K-HOrderDNN to tackle more complex PDEs and to address high dimensional sampling issues.
\section*{Acknowledgments}
This research is supported partly by National Natural Science Foundation of China with grants 12101609, National Key R\&D Program of China with grants 2019YFA0709600, 2019YFA0709602,  2020YFA0713500, and Hunan Provincial Innovation Foundation For Postgraduate with grants CX20220646.

\bibliographystyle{elsarticle-num}
\bibliography{zyq}

\appendix
\renewcommand\thesection{A}
\newpage
\section*{Appendix A. Proof of Theorem \ref{Theorem4}}
\label{sec:appendix_a}
\setcounter{thm}{0}
\setcounter{equation}{0}
\setcounter{definition}{0}
\renewcommand{\thethm}{A\arabic{thm}}
\renewcommand{\theequation}{A\arabic{equation}}
\renewcommand{\thedefinition}{A\arabic{definition}}

\begin{proof}
Let $g$ be the K-outer function associated with $f$, whose smoothness is characterized by the standard modulus of continuity. Owing to the uniform continuity of $g$, there exists a linear spline $S_g$ over an equally partitioned sequence with uniform spacing $\delta=\frac{d}{n}$ such that
\begin{equation}
|g(z)-S_g(z)|\le \omega \left ( g, \delta \right ), \forall  z \in [0,d],
\label{equationA1}
\end{equation}
for any $\delta > 0 $. It is easy to confirm that Equation \eqref{equationA1} holds by proving its applicability within each subinterval $z \in [z_i,z_{i+1}]$. 
\par
From equation \eqref{KST} in Theorem \ref{thm1}, we conclude that
\par
\begin{align}
\left| f(x) - \sum_{q=0}^{2d} S_g \left( \sum_{i=1}^{d} \lambda_{i} \phi_q(x_{i}) \right) \right| \leq \sum_{q=0}^{2d} \left| g \left( \sum_{i=1}^{d} \lambda_{i} \phi_q(x_{i}) \right) - S_g \left( \sum_{i=1}^{d} \lambda_{i} \phi_q(x_{i}) \right) \right|
\leq (2d+1)\omega(g, \delta).
\label{equationA2}
\end{align}
Building on the definition of $\widetilde{L}_{q}$ given in Theorem \ref{thm2}, and applying equation \eqref{eq2}, we get
\par
\begin{align}
|\sum_{i=1}^{d} \lambda _{i}  \phi _{q}\left ( x_{i}  \right )   - \sum_{i=1}^{d} \lambda _{i}  \widetilde{L}_{q}\left ( x_{i}  \right  )| \le \frac{d K C_\phi}{p^{\alpha}}.
\label{equationA3}
\end{align}
Thus, by using the triangle inequality, the estimates from \eqref{equationA1} and \eqref{equationA3}, and the property of the modulus of continuity, we deduce that
\begin{equation}
\begin{aligned}
|S_g\Big(\sum_{i=1}^{d} \lambda _{i} \phi _{q}(x_{i})\Big)-S_g\Big(\sum_{i=1}^{d}\lambda _{i} \widetilde{L}_{q}(x_{i})\Big)|&\le 2\omega(g,\delta) + \Big|g\Big(\sum_{i=1}^{d} \lambda _{i} \phi _{q}(x_{i})\Big)-g\Big(\sum_{i=1}^{d} \lambda _{i} \widetilde{L}_{q}(x_{i})\Big)\Big| \\
&\le 2\omega(g,\delta) + \omega(g, \frac{d K C_\phi}{p^{\alpha}}) \\
& \le 2d\omega(g,\frac{1}{n}) + (d K C_\phi +1)\omega(g,\frac{1}{p^{\alpha}}).
\end{aligned}
\label{equationA4}
\end{equation}

By integrating the above estimations, we have
\begin{align*}
\left | f(x)-\sum_{q=0}^{2d} S_g\left ( \sum_{i=1}^{d} \lambda _{i}  \widetilde{L}_{q}\left ( x_{i}  \right )  \right )  \right | &\le (2d+1)\left ( d\omega \left ( g, \frac{1}{n} \right )+2d\omega(g,\frac{1}{n}) + (d K C_\phi +1)\omega(g,\frac{1}{p^{\alpha}}) \right ) \\
&\le d(2d+1)\left ( 3\omega \left ( g, \frac{1}{n} \right ) +(K C_\phi +1)\omega(g,\frac{1}{p^{\alpha}}) \right ) .
\end{align*}
This completes the proof of Theorem \ref{Theorem4}.
\end{proof}

\section*{Appendix B. Results on various parameter settings}
\setcounter{table}{0} 
\renewcommand{\thetable}{B\arabic{table}}
Table \ref{tab:B2} investigate the effect of increasing $gw$ and $gd$ on error reduction. Table \ref{tab:B1} demonstrate the effect of $p$ and changes of relative errors with respect to $hw$ under different $hd$ conditions.

\begin{table}[H]
	\centering
	\caption{The relative $L_2$ errors, along with the $gw$, obtained by K-HOrderDNNs under different $gd$ settings for the problem \eqref{equation2} when $d$=2.}
	\setlength{\tabcolsep}{3pt}
	\resizebox{0.73\columnwidth}{!}{
	\begin{tabular}{cccccc}
		\toprule
		Model & $p$ & \diagbox{$gd$}{$gw$}  & 45    & 120   & 180 \\
		\midrule
		\multirow{15}[30]{*}[7ex]{K-HOrderDNN} & \multirow{3}[6]{*}[1.5ex]{1} & 2     & 5.21E-01 & 4.37E-01 & 7.33E-01\\
		&       & 3     & 4.81E-01 & 3.66E-01 & \textbf{1.37E-01}\\
		&       & 4     & 3.46E-01 & 3.19E-01 & 3.02E-01\\
		\cmidrule(lr){2-6}          & \multirow{3}[6]{*}[1.5ex]{3} & 2     & 1.77E-01 & 1.47E-01 & 9.72E-02\\
		&       & 3     & 1.48E-01 & 6.19E-02 & 6.25E-02\\
		&       & 4     & 9.87E-02 & 5.44E-02 & \textbf{3.06E-02}\\
		\cmidrule(lr){2-6}          & \multirow{3}[6]{*}[1.5ex]{5} & 2     & 7.41E-02 & 1.43E-01 & \textbf{2.48E-02}\\
		&       & 3     & 5.10E-02 & 4.23E-02 & 4.69E-02\\
		&       & 4     & 3.47E-02 & 3.47E-01 & 3.85E-01\\
		\cmidrule(lr){2-6}          & \multirow{3}[6]{*}[1.5ex]{7} & 2     & 2.16E-02 & 1.24E-02 & 2.04E-02\\
		&       & 3     & 3.09E-02 & 1.43E-02 & 8.29E-03\\
		&       & 4     & 3.31E-02 & 1.27E-02 & \textbf{8.11E-03}\\
		\cmidrule(lr){2-6}          & \multirow{3}[6]{*}[1.5ex]{9} & 2     & 2.13E-02 & 1.17E-02 & 1.13E-02\\
		&       & 3     & 9.85E-03 & 1.13E-02 & 1.02E-02\\
		&       & 4     & 2.31E-02 & 8.31E-03 & \textbf{6.04E-03}\\
		\bottomrule
	\end{tabular}%
	}
	\label{tab:B2}%
\end{table}%
\par
\begin{table}[H]
\centering
\setlength{\tabcolsep}{3pt}
\caption{Comparison of the relative $L_2$ errors among PINN, HOrderDNNs, and K-HOrderDNNs on the problem \eqref{equation2} when $d$=2.}
\resizebox{0.73\columnwidth}{!}{
\begin{tabular}{ccccccc}
	\toprule%
	Model & $p$ & \diagbox{$hd$}{$hw$} & 5    & 10   & 25   & 45 \\ \midrule
	\multirow{3}[6]{*}[1.5ex]{PINN} & \multirow{3}[6]{*}[1.5ex]{\textbackslash{}}& 1     & 7.62E-01 & 7.87E-01 & 4.96E-01 & 3.77E-01 \\
	\multicolumn{2}{c}{} & 2     & 7.19E-01 & 7.14E-01 & 4.38E-01 & 2.71E-01 \\
	\multicolumn{2}{c}{} & 3     & 6.85E-01 & 6.36E-01 & 3.80E-01 & \textbf{2.13E-01} \\ 
	\midrule
	\multirow{15}[30]{*}[7ex]{HOrderDNN} & \multirow{3}[6]{*}[1.5ex]{1} & 1     & 6.62E-01 & 6.18E-01 & 4.95E-01 & 3.26E-01 \\
	&       & 2     & 6.70E-01 & 6.26E-01 & 3.64E-01 & 1.77E-01 \\
	&       & 3     & 6.47E-01 & 6.24E-01 & 2.98E-01 & \textbf{1.11E-01} \\
	\cline{2-7}          & \multirow{3}[6]{*}[1.5ex]{3} & 1     & 5.48E-01 & 4.70E-01 & 2.44E-01 & 9.83E-02 \\
	&       & 2     & 5.43E-01 & 3.60E-01 & 1.61E-01 & 5.69E-02 \\
	&       & 3     & 4.96E-01 & 4.26E-01 & 9.94E-02 & \textbf{3.20E-02} \\
	\cline{2-7}          & \multirow{3}[6]{*}[1.5ex]{5} & 1     & 4.74E-01 & 3.22E-01 & 1.08E-01 & 3.85E-02 \\
	&       & 2     & 4.60E-01 & 2.62E-01 & 7.51E-02 & 2.60E-02 \\
	&       & 3     & 4.35E-01 & 2.42E-01 & 5.86E-02 & \textbf{1.71E-02} \\
	\cline{2-7}          & \multirow{3}[6]{*}[1.5ex]{7} & 1     & 3.79E-01 & 2.50E-01 & 7.21E-02 & 2.49E-02 \\
	&       & 2     & 3.76E-01 & 2.46E-01 & 4.56E-02 & 1.63E-02 \\
	&       & 3     & 4.15E-01 & 1.88E-01 & 4.01E-02 & \textbf{1.33E-02} \\
	\cline{2-7}          & \multirow{3}[6]{*}[1.5ex]{9} & 1     & 3.22E-01 & 1.87E-01 & 5.18E-02 & 1.83E-02 \\
	&       & 2     & 3.16E-01 & 1.75E-01 & 3.50E-02 & 1.29E-02 \\
	&       & 3     & 2.94E-01 & 1.18E-01 & 2.93E-02 & \textbf{1.06E-02} \\ \midrule
	\multirow{15}[30]{*}[7ex]{K-HOrderDNN} & \multirow{3}[6]{*}[1.5ex]{1} & 1     & 7.31E-01 & 6.72E-01 & 5.44E-01 & 4.24E-01 \\
	&       & 2     & 7.41E-01 & 9.91E-01 & 5.52E-01 & 3.46E-01 \\
	&       & 3     & 7.62E-01 & 6.87E-01 & 3.25E-01 & \textbf{3.07E-01} \\
	\cline{2-7}          & \multirow{3}[6]{*}[1.5ex]{3} & 1     & 6.03E-01 & 4.49E-01 & 1.29E-01 & 1.29E-01 \\
	&       & 2     & 5.16E-01 & 4.47E-01 & 1.16E-01 & 2.96E-02 \\
	&       & 3     & 4.19E-01 & 2.43E-01 & 7.83E-02 & \textbf{1.50E-02} \\
	\cline{2-7}          & \multirow{3}[6]{*}[1.5ex]{5} & 1     & 2.15E-01 & 2.13E-01 & 4.62E-02 & 3.05E-01 \\
	&       & 2     & 3.08E-01 & 3.35E-01 & 2.39E-02 & 3.23E-01 \\
	&       & 3     & 2.88E-01 & 8.49E-02 & 2.11E-02 & \textbf{9.37E-03} \\
	\cline{2-7}          & \multirow{3}[6]{*}[1.5ex]{7} & 1     & 2.33E-01 & 9.44E-02 & 2.76E-02 & 1.71E-02 \\
	&       & 2     & 1.41E-01 & 4.83E-02 & 1.53E-02 & 6.76E-03 \\
	&       & 3     & 3.21E-01 & 4.90E-02 & 1.41E-02 & \textbf{4.65E-03} \\
	\cline{2-7}          & \multirow{3}[6]{*}[1.5ex]{9} & 1     & 2.07E-01 & 7.57E-02 & 2.21E-02 & 1.25E-02 \\
	&       & 2     & 1.40E-01 & 5.01E-02 & 9.96E-03 & 6.60E-03 \\
	&       & 3     & 1.58E-01 & 3.05E-02 & 9.90E-03 & \textbf{3.97E-03} \\
	\bottomrule
\end{tabular}%
}
\label{tab:B1}%
\end{table}%

\section*{Appendix C. Comparison of KAN and K-HOrderDNN}
\setcounter{table}{0} 
\renewcommand{\thetable}{C\arabic{table}}
\setcounter{figure}{0} 
\renewcommand{\thefigure}{C\arabic{figure}}
In this section, we provide a detailed comparison between KAN and our K-HOrderDNN, both in terms of architecture and performance on 10D high-frequency function fitting  and 2D high-frequency Poisson equation-solving tasks.
\subsection*{C.1 Architectural Comparison}
Ziming Liu et al.\cite{liu2024kan} recently proposed Kolmogorov Arnold Networks (KAN) for function approximation. Both KAN and K-HOrderDNN are inspired by the Kolmogorov-Arnold Representation Theorem(KAT) or KST, but they differ in how they model univariate K-inner and K-outer functions. In KAN, the univariate K-inner and K-outer functions are parametrized as B-spline curves, with learnable coefficients for the local B-spline basis functions. However, the prototype of KAN is relatively simple, limiting its ability to approximate complex functions, especially when smooth splines are involved. To address this, KAN is extended to deeper and wider architectures to enhance its approximation capabilities. In contrast, K-HOrderDNN approximates the univariate K-inner functions using HOrderDNNs and models the univariate K-outer function with a fully connected subnetwork. Moreover, K-HOrderDNN employs fully connected subnetworks with fixed activation functions at the nodes and learnable weights on the edges. In contrast, KAN uses learnable activation functions at the edges (parameterized by splines) while the nodes perform only summation operations.

\subsection*{C.2 High-Dimensional Function Fitting}
We compare the performance of KAN and K-HOrderDNNs on a 10-dimensional high-frequency function fitting problem \eqref{equation3} using the hyperparameters described in Subsection \ref{f4}. KAN, with a structure of [10, 50, 1] and a grid size of 100, performs comparably to K-HOrderDNN when the number of parameters is comparable, as shown in Table \ref{tab:addlabel}. However, the training time of KAN is about 12 times slower than that of K-HOrderDNN.
\begin{table}[htbp]
	\centering
	\caption{The relative $L_2$ errors and the number of parameters obtained by KAN on problem \eqref{equation3} when $d$ = 10.}
	\begin{tabular}{ccc}
		\toprule
		\multirow{2}[4]{*}{Method} & \multicolumn{2}{c}{d=10} \\
		\cmidrule{2-3}          & Params & RLE \\
		\midrule
		KAN   & 9.33E+04 & 3.36E-02 \\
		K-HOrderDNN(9) & 9.56E+04 & 2.78E-02 \\
		\bottomrule
	\end{tabular}%
	\label{tab:addlabel}%
\end{table}%
\subsection*{C.3 2D Poisson Equation Solving}
We compare the performance of KAN and K-HOrderDNN on a 2D high-frequency poisson problem \eqref{equation4}. The hyperparameter configurations for KAN include network structures of [2, 20, 1], [2, 40, 1], [2, 60, 1], and [2, 80, 1], with corresponding grid sizes of 60, 72, 80, and 100, and $\beta = 10,000$. Other hyperparameters follow the settings outlined in Subsection \ref{case2d}. The relative $L_2$ errors and the number of parameters for each configuration are summarized in Table \ref{tab:R5}. Among these configurations, the best performance is achieved with the KAN structure [2, 60, 1] and grid size 60, where the model reaches an relative $L_2$ errors of 5.09E-02 with 2.38E+04 parameters. However, this result is still inferior to that of K-HOrderDNN, which achieves an relative $L_2$ errors of 1.91E-02 with 1.41E+04 parameters (Table \ref{tab:table7}). Additionally, KAN requires about 5 times longer training than K-HOrderDNN. In other configurations, the performance of KAN is significantly worse despite extensive tuning of hyperparameters, including $\beta$. Moreover, as shown in Table \ref{tab:R5}, we observe that the error does not exhibit a clear trend with respect to variations in grid size or network width, thereby making it challenging to determine the optimal parameter values.

\begin{table}[H]
	\centering
	\caption{The relative $L_2$ errors and the number of parameters obtained by KAN on Possion problem (4.4) under different  structure of KAN and grid size.}
	\begin{tabular}{cccccc}
		\toprule
		\multirow{2}[4]{*}{The structure of KAN} & \multirow{2}[4]{*}{} & \multicolumn{4}{c}{Grid size} \\
		\cmidrule{3-6}          &       & \multicolumn{1}{c}{60} & \multicolumn{1}{c}{72} & \multicolumn{1}{c}{80} & \multicolumn{1}{c}{100} \\
		\midrule
		\multirow{2}[1]{*}{[2, 20, 1] } & Params & \multicolumn{1}{c}{7.94E+03} & \multicolumn{1}{c}{9.38E+03} & \multicolumn{1}{c}{1.03E+04} & \multicolumn{1}{c}{1.27E+04} \\
		& RLE   & \multicolumn{1}{c}{1.02E-01} & \multicolumn{1}{c}{4.57E-01} & \multicolumn{1}{c}{8.41E-01} & \multicolumn{1}{c}{8.02E-01} \\
		\multirow{2}[0]{*}{[2, 40, 1] } & Params & 1.59E+04 & 1.88E+04 & 2.07E+04 & 2.55E+04 \\
		& RLE   & 7.98E-02 & 2.69E-01 & 8.77E-01 & 8.95E-01 \\
		\multirow{2}[0]{*}{[2, 60, 1] } & Params & 2.38E+04 & 2.81E+04 & 3.10E+04 & 3.82E+04 \\
		& RLE   & \textbf{5.09E-02} & 2.77E-01 & 8.34E-01 & 8.65E-01 \\
		\multirow{2}[1]{*}{[2, 80, 1] } & Params & 3.18E+04 & 3.75E+04 & 4.14E+04 & 5.10E+04 \\
		& RLE   & 6.25E-02 & 3.85E-01 & 9.01E-01 & 8.87E-01 \\
		\bottomrule
	\end{tabular}%
	\label{tab:R5}%
\end{table}%

\section*{Appendix D. Comparison of TNN and K-HOrderDNN}
\setcounter{table}{0} 
\renewcommand{\thetable}{D\arabic{table}}
\setcounter{figure}{0} 
\renewcommand{\thefigure}{D\arabic{figure}}
\setcounter{equation}{0}
\renewcommand{\theequation}{D\arabic{equation}}
In this section, we provide a detailed comparison between TNN and our K-HOrderDNN in terms of architecture and performance on high-dimensional problems involving both tensor-product-type and non-tensor-product-type source terms.
\subsection*{D.1 Architectural Comparison}
Hehu Xie et al.\cite{wang2024solving} recently introduced the Tensor Neural Network (TNN), which demonstrates great potential for solving high-dimensional problems.
Both K-HOrderDNN and TNN are designed to approximate a high-dimensional function, but they adopt fundamentally different network architectures. TNN approximates this high-dimensional function by leveraging a low-rank tensor structure composed of one-dimensional functions $\phi_i(x_i)$, which comes from $d$ subnetworks with one-dimensional input and multidimensional output, where $d$ denotes the dimensionality of the problem. Mathematically, it is expressed as:  
\begin{equation}
	u(x_1, x_2, \cdots, x_d; \theta) = \sum_{j=1}^p c_j \prod_{i=1}^d \phi_{i,j}(x_i; \theta_i),
\end{equation}
where each term $ \prod_{i=1}^d \phi_{i,j}(x_i; \theta_i)$ represents a "rank-one" component formed by the tensor product of one-dimensional functions.   
In contrast, K-HOrderDNN leverages the Kolmogorov Superposition Theorem (KST), which decomposes multivariate functions into compositions and sums of simpler univariate functions. It approximates the K-inner univariate functions using HOrderDNNs and the K-outer function with a fully connected subnetwork. These components are combined within the KST framework to represent a high-dimensional function.

\subsection*{D.2 A high-dimensional Poisson problem with tensor-product-type source term}
In this experiment, we compare TNN and K-HOrderDNN on a high-dimensional Poisson problem with a tensor-product-type source term, as described in Section 5.1 of \cite{wang2024solving}. Specifically, we consider a $d$-dimensional Poisson problem \eqref{equation4} with the exact solution: 		 
\begin{equation}
	u(x_1, x_2, \cdots, x_d) = \sum_{k=1}^d \sin \left(2 \pi x_k\right) \cdot \prod_{i \neq k}^d \sin \left(\pi x_i\right)
	\label{Req2}%
\end{equation}
with the right-hand side set to $(d+3)\pi^2u$ and a homogeneous Dirichlet boundary condition.
\par
In this example, for K-HOrderDNNs, the hyperparameters are set as follows: for $ d = 5 $, $ hw = gw = 105 $, $ hd = 1 $, and $ gd = 2 $; and for $ d = 10 $, $hw=gw=210$, $ hd = 1 $, and $ gd = 2 $. For TNN, we use the same hyperparameter settings from \cite{wang2024solving}. In each epoch, we resample 8000 training points for $ N_f $ and 2000 training points for $ N_b $.
\par
Table \ref{tab:R6} summarizes the number of parameters, the relative $L_2$  error, and average computation time per 100 steps on GPU for both TNN and K-HOrderDNN for dimension $d = 5$ and $ d = 10$. 	  
As we can see, TNN shows superior accuracy while requiring fewer computational time than K-HOrderDNNs. This result aligns with expectations because TNN achieves efficient computation by leveraging its low-rank property, which reduces high-dimensional integrals in the loss function to one-dimensional integrals. These can then be effectively computed using classical quadrature schemes. In contrast, K-HOrderDNN is based on the PINN framework, where the loss function calculation relies on high-dimensional integrals. Accurately and efficiently computing these integrals is challenging due to the slow convergence of Monte Carlo methods and their reliance on sampling techniques. However, the effectiveness of TNN is limited to the high-dimensional problems where the coefficients and source terms exhibit a tensor product structure.
\begin{table}[H]
	\centering
	\caption{The relative $ L_2 $ errors, number of parameters, and average computation time per 100 steps on GPU for each method on problem \eqref{equation4}, with the exact solution given in equation \eqref{Req2}, are presented for dimensions $ d = 5 $ and $ d = 10$.}
	\begin{tabular}{ccccrrrr}
		\toprule
		\multirow{2}[4]{*}{Method} & \multirow{2}[4]{*}{$p$} & \multicolumn{3}{c}{d=5} & \multicolumn{3}{c}{d=10} \\
		\cmidrule{3-8}          &       & Params & RLE   & \multicolumn{1}{c}{A100} & \multicolumn{1}{c}{Params} & \multicolumn{1}{c}{RLE} & \multicolumn{1}{c}{A100} \\
		\midrule
		\multirow{5}[2]{*}{K-HOrderDNN} & 1     & 2.00E+04 & 2.54E-02 & 3.1s & 9.39E+04 & 6.17E-02 & 8.8s \\
		& 3     & 2.03E+04 & 6.00E-03 & 3.9s & 9.43E+04 & 4.17E-02 & 11.7s \\
		& 5     & 2.05E+04 & 4.59E-03 & 5.8s & 9.47E+04 & 3.21E-02 & 14.2s \\
		& 7     & 2.07E+04 & 4.34E-03 & 7.2s & 9.52E+04 & 9.85E-03 & 17.2s \\
		& 9     & 2.09E+04 & 3.12E-03 & 8.8s & 9.56E+04 & 1.99E-02 & 21.3s \\
		\midrule
		TNN   & \textbackslash{} & 1.27E+05 & 3.45E-07 & 0.99s & 2.55E+05 & 4.84E-07 & 1.88s \\
		\bottomrule
	\end{tabular}%
	\label{tab:R6}%
\end{table}%
\subsection*{D.3 A high-dimensional Poisson problem with non-tensor-product-type source term}
In this experiment, the exact solution \eqref{eq2} is modified as follows:
\begin{equation}
u(x_1, x_2, \cdots, x_d)=\sum_{i=1}^{d-1} \sin \left(16 \pi x_i x_{i+1}\right)
\label{eq3}%
\end{equation}
In this example, for K-HOrderDNNs, the hyperparameters are set as follows: for $ d = 5 $, $ hw = gw = 220 $, $ hd = 1 $, and $ gd = 2 $; and for $ d = 10 $, $hw=gw=315$, $ hd = 1 $, and $ gd = 2 $. In each epoch, we resample 8000 training points for $ N_f $ and 2000 training points for $ N_b $.

It is easy to verify that the source term on problem \eqref{equation4} is not of tensor-product form. The relative $L_2$ errors and the number of parameters for both PINN and K-HOrderDNNs applied to problem \eqref{equation4} are shown in Table \ref{tab:R7}. In this case, TNN cannot work, while both PINN and K-HOrderDNNs remain effective, with K-HOrderDNN achieving lower relative $L_2$ error than PINN. It is worth noting that Hehu Xie et al.\cite{li2024tensor} recently proposed an interpolation method employing a tensor neural network to approximate a non-tensor-product-type high-dimensional function $f(x)$. This approach enables the use of the approximated function to compute high-dimensional integrals of  $f(x)$. However, finding an approximation for a high dimensional function is generally not easier than computing the integrations of a high dimensional function\cite{li2024tensor}. 
\begin{table}[H]
\centering
\caption{The number of parameters and the relative $L_2$ errors obtained by PINN, and KHOrderDNNs for problem \eqref{equation4} with a solution represented by Eq. \eqref{eq3}. Here, for $ d = 5 $, $ hw = gw = 220 $, $ hd = 1 $, and $ gd = 2 $; and for $ d = 10 $, $hw=gw=315$, $ hd = 1 $, and $ gd = 2 $.}
\begin{tabular}{cccccc}
	\toprule
	\multirow{2}[4]{*}{Method} & \multirow{2}[4]{*}{$p$} & \multicolumn{2}{c}{d=5} & \multicolumn{2}{c}{d=10} \\
	\cmidrule{3-6}          &       & Params & RLE   & Params & RLE \\
	\midrule
	PINN  & \textbackslash{} & 1.47E+05 & 3.17E-01 & 3.02E+05 & 9.39E-01 \\
	\multirow{5}[1]{*}{K-HOrderDNN} & 1     & 6.43E+04 & 2.54E-02 & 1.739E+05 & 8.82E-01 \\
	& 3     & 6.47E+04 & 6.90E-02 & 1.745E+05 & 2.04E-01 \\
	& 5     & 6.51E+04 & 4.67E-02 & 1.752E+05 & 1.26E-01 \\
	& 7     & 6.56E+04 & 7.05E-02 & 1.758E+05 & 8.85E-02 \\
	& 9     & 6.60E+04 & 6.07E-02 & 1.764E+05 & 8.85E-02 \\
	\bottomrule
\end{tabular}%
\label{tab:R7}%
\end{table}%

\section*{Appendix E. Proof of Theorem \ref{thm_tanh}}
\label{sec:appendix_a}
\setcounter{thm}{0}
\setcounter{lem}{0}
\setcounter{equation}{0}
\setcounter{definition}{0}
\renewcommand{\thethm}{E\arabic{thm}}
\renewcommand{\thelem}{E\arabic{lem}}
\renewcommand{\theequation}{E\arabic{equation}}
\renewcommand{\thedefinition}{E\arabic{definition}}
We now extend the approximation result to the Tanh activation function $ \sigma_1$ for a dense subclass of continuous functions. For any continuous function $f$, we also use polynomials $L_{q}(x)$ to approximate the K-inner functions $\phi_q(x)$. By Theorem \ref{thm1}, $\phi_q \in Lip_{C_\phi}(\alpha)$ with $C_\phi$ being their common Lipschitz constant and $\alpha = \log_{10}2$. According to Lemma \ref{Jackson}, there exist  polynomials $L_q \in \mathcal{Q}_p (\mathbb{R})$ such that 
\begin{equation}
	\max_{x\in [0,1]}\left|\phi_q(x)-L_q(x)\right|\le  \frac{KC_\phi}{p^{\alpha}} \colon = \Delta,\quad q=0,1,\cdots ,2d.
	\label{delta}
\end{equation}
It is easy to verify that $ L_q(x) \in [-\Delta, 1+\Delta] $ using equation \eqref{delta} and the definition of $\phi_q(x)$.

Before presenting the proof, we first state a variant of Corollary 5.4 from \cite{de2021approximation} (Lemma \ref{lemm3:tanh}), followed by Lemma \ref{tanh_max_eq}, which is used to establish an upper bound for the approximation of the K-inner function.

\begin{lem} Let $\Omega = [0, d]$, and suppose $g : \Omega \to \mathbb{R}$ is a Lipschitz continuous function with Lipschitz constant $C_g$. Then, for $ N \in \mathbb{N} $ with $ N > 5 $, there exists a tanh neural network $ \hat{g}^N $  with two hidden layers: one of width at most $ (N - 1) $ and the other of width at most $ 6N$, such that
	\begin{equation}
		\| g - \hat{g}^N \|_{L^\infty(\Omega)} \leq \frac{7d C_g}{N}.
		\label{lemm3:eq}
	\end{equation}
	\label{lemm3:tanh}
\end{lem}
\begin{proof}
	Define $f(x) = g(xd) $ for $ x \in [0, 1] $. It is straightforward to verify that $ f(x) $ satisfies a Lipschitz condition with constant $ d C_g $. By using the one-dimensional case of Corollary 5.4 in \cite{de2021approximation}, there exists a tanh neural network $ \hat{f}^N $ with two hidden layers, where the layer widths are at most $ (N - 1) $ and $ 6N $, such that
	\begin{equation*}
		\| f(x) - \hat{f}^N(x; \theta) \|_{L^\infty([0,1])} \leq \frac{7d C_g}{N}.
	\end{equation*}
	By  defining \( \hat{g}^N(t; \theta) = \hat{f}^N(t/d; \theta) \), we obtain
	\begin{equation*}
		\| g(t) - \hat{g}^N(t; \theta) \|_{L^\infty([0, d])} \leq \frac{7d C_g}{N}.
	\end{equation*}
\end{proof}
\begin{lem} Let \( g(t) = w \tanh\left( \frac{t + \Delta}{w(1 + 2\Delta)} \right) - t \), where \( t \in [-\Delta, 1 + \Delta] \), \( \Delta > 0 \), and \( w > 0 \). There exists a constant \( w_0 > 0 \) such that for all \( w \geq w_0 \), the maximum absolute value of \( g(t) \) on the interval \( t \in [-\Delta, 1 + \Delta] \) satisfies:
	\begin{equation}
		| g(t)| \le  2\Delta.
		\label{tanh_max_eq}
	\end{equation}
	\label{lemm3:tanh_max}
\end{lem}
\begin{proof}
The function $g(t)$ is smooth and differentiable on the interval \( t \in [-\Delta, 1 + \Delta] \), and its derivative is given by:
\begin{equation}
g'(t) = w \cdot \text{sech}^2\left( \frac{t + \Delta}{w(1 + 2\Delta)} \right) \cdot \frac{1}{w(1 + 2\Delta)} - 1 \le 0.
\end{equation}
Since $g(t)$ is monotonically decreasing on this interval, with the value of \( g(-\Delta) = \Delta > 0 \) and \( g(1 + \Delta) = w \tanh\left( \frac{1}{w} \right) - (1 + \Delta) < 0 \), it follows that the maximum value of \( |g(t)| \) occurs at either \( t = -\Delta \) or \( t = 1 + \Delta \). Thus, we have:
\begin{equation}
|g(t)| \leq \max \left( \Delta, 1 + \Delta - w \tanh\left( \frac{1}{w} \right) \right).
\end{equation}
Next, consider the function \( 1 - w \tanh\left( \frac{1}{w} \right) \), which is continuous and monotonically decreasing for \( w > 0 \), with the limit \( 1 - w \tanh\left( \frac{1}{w} \right) \to 0 \) as \( w \to 0^+ \). Therefore, there exists a constant \( w_0 > 0 \) such that for all \( w \geq w_0 \), we have:
\begin{equation}
1 - w \tanh\left( \frac{1}{w} \right) \leq \Delta.
\end{equation}
Consequently, for \( w \geq w_0 \), we obtain:
\begin{equation}
|g(t)| \leq \max \left( \Delta, 2\Delta \right) = 2\Delta.
\end{equation}
\end{proof}
\par
To restrict the values of $L_q(x)$ within the interval $[0, 1]$, we define $\bar{L}_q(x)$ as
\begin{equation}
	\bar{L}_q(x) = w\sigma_1 \left(\frac{L_q(x)+\Delta}{w(1+2\Delta)}\right)= w\sigma_1 \left(\frac{\sum_{j=1}^{p+1} a_{j,q} \psi_j(x)+\Delta}{w(1+2\Delta)}\right) \approx \phi_q(x), \quad q = 0, 1, \dots, 2d.
	\label{req:1}
\end{equation}
which means the subnetwork $h_p$ in K-HOrderDNN is chosen as a high order multi-layer neural network with one nonlinear transformation layer and two hidden layers with $2d+1$ neurons in each hidden layer. To approximate the K-outer function $g$, we utilize a tanh neural network $\hat{g}^N(t; \theta)$ with two hidden layers, where the layer widths are at most $ (N - 1) $ and $ 6N $, as presented in Lemma \ref{lemm3:tanh}.
\par
Thus, we obtain an approximation of the original function $f(x_1, \ldots, x_d)$ as follows:
\begin{equation*}
	f(x_1,x_2,\cdots,x_d)\approx \sum_{q=0}^{2d} \hat{g}^N \left(\sum_{i=1}^{d} \lambda_{i} \bar{L}_{q}(x_i)\right).
\end{equation*}
\par
We now proceed to prove the result.
\begin{proof}
	The approximation error can be bounded as follows by applying the triangle inequality:
	\begin{equation}
		\begin{aligned} 
			\left | f(x)-\sum_{q=0}^{2d} \hat{g}^N \left ( \sum_{i=1}^{d} \lambda _{i}  \bar{L}_{q}\left ( x_{i}  \right )  \right )  \right | & \leqslant  
			\sum_{q=0}^{2d} \left| g\left( \sum_{i=1}^d \lambda_i \phi_q(x_i) \right) - \hat{g}^N\left( \sum_{i=1}^d \lambda_i \phi_q(x_i) \right) \right| \\
			& \quad + \sum_{q=0}^{2d} \left| \hat{g}^N \left( \sum_{i=1}^d \lambda_i \phi_q(x_i) \right) - \hat{g}^N \left( \sum_{i=1}^d \lambda_i \bar{L}_q(x_i) \right) \right|
		\end{aligned}
		\label{rr1}
	\end{equation}
	We now proceed with an individual approximation for each term on the right-hand side.
	\noindent 
	\par
	\textbf{Step 1:} First term of \eqref{rr1}. By using Lemma \ref{lemm3:tanh}, we have 
	\begin{equation}
		\begin{aligned} 
			\sum_{q=0}^{2d} \left| g\left( \sum_{i=1}^d \lambda_i \phi_q(x_i) \right) - \hat{g}^N\left( \sum_{i=1}^d \lambda_i \phi_q(x_i) \right) \right| 
			\leq \frac{7d(2d+1) C_g}{N}
		\end{aligned}
	\end{equation}
	\par
	\noindent \textbf{Step 2:} Second term of \eqref{rr1}. It is easy to verify that \( |\phi_{q}(x) - \bar{L}_{q}(x)| \leq 3 |\phi_{q}(x) - L_{q}(x)| \) by applying the triangle inequality, the definition of \(\bar{L}_{q}(x)\), and Lemma \ref{tanh_max_eq}. Hence, we have
	\begin{align}
		|\sum_{i=1}^{d} \lambda _{i}  \phi _{q}\left ( x_{i}  \right )   - \sum_{i=1}^{d} \lambda _{i}  \bar{L}_{q}\left ( x_{i}  \right  )| \le \frac{3d KC_\phi}{p^{\alpha}}.
		\label{R:equationA3}
	\end{align}
	
	Thus, by using the triangle inequality, the estimates from Lemma \ref{lemm3:tanh} and \eqref{R:equationA3}, we deduce that
	\begin{equation}
		\begin{aligned}
			|\hat{g}\Big(\sum_{i=1}^{d} \lambda _{i} \phi _{q}(x_{i})\Big)-\hat{g}\Big(\sum_{i=1}^{d}\lambda _{i} \bar{L}_{q}(x_{i})\Big)|&\le 2\frac{7d C_g}{N} + \Big|g\Big(\sum_{i=1}^{d} \lambda _{i} \phi _{q}(x_{i})\Big)-g\Big(\sum_{i=1}^{d} \lambda _{i} \bar{L}_{q}(x_{i})\Big)\Big| \\
			&\le \frac{14d C_g}{N} +  \frac{C_g 3d KC_\phi}{p^{\alpha}}\\
			& = d C_g(\frac{14}{N}+\frac{3KC_\phi}{p^{\alpha}}).
		\end{aligned}
		\label{R:equationA4}
	\end{equation}
	\par
	\noindent \textbf{Step 3:} Final error bound. Combining the contributions from the two terms of \eqref{rr1} then proves that
	\begin{equation*}
		\left | f(x)-\sum_{q=0}^{2d} \hat{g}^N\left ( \sum_{i=1}^{d} \lambda _{i}  \bar{L}_{q}\left ( x_{i}  \right )  \right )  \right | 
		\le C_g(2d+1)d (\frac{21}{N}+\frac{3KC_\phi}{p^{\alpha}}).
	\end{equation*}
	This completes the proof.
\end{proof}
\end{document}